\documentclass{article}
\usepackage{amssymb}
\usepackage{amsmath}
\usepackage{mathtools}
\usepackage{enumerate}
\usepackage{esint}
\usepackage{siunitx}
\usepackage{fullpage}
\usepackage{graphicx}
\usepackage{caption}
\usepackage{subcaption}
\usepackage{wrapfig}
\usepackage{epstopdf}
\usepackage{float}

\newcommand{\newpar}{\vspace{5mm}\par}
\newcommand{\vnorm}[1]{\left\|#1\right\|}
\usepackage{amsthm}
\usepackage{units}
\usepackage{tikz}
\usepackage{bbm}
\usepackage{scalefnt}
\usepackage{setspace}

\newtheorem{theorem}{Theorem}[section]
\newtheorem{proposition}[theorem]{Proposition}
\newtheorem{lemma}[theorem]{Lemma}
\newtheorem{corollary}[theorem]{Corollary}

\begin{document}
\title{Linked systems of symmetric designs}
\author{Brian G. Kodalen \\
	Department of Mathematical Sciences \\
	Worcester Polytechnic Institute \\
	Worcester, Massachusetts \\
	{\tt bgkodalen@wpi.edu}}
\date{\today}
\maketitle
\medskip
\begin{abstract}
	A linked system of symmetric designs (LSSD) is a $w$-partite graph ($w\geq 2$) where the incidence between any two parts corresponds to a symmetric design and the designs arising from three parts are related. The original construction for LSSDs by Goethals used Kerdock sets, in which $v$ is a power of two. Some four decades later, new examples were given by Davis et.\ al.\ and Jedwab et.\ al.\ using difference sets, again with $v$ a power of two. In this paper we develop a connection between LSSDs and ``linked simplices", full-dimensional regular simplices with two possible inner products between vertices of distinct simplices. We then use this geometric connection to construct sets of equiangular lines and to find an equivalence between regular unbiased Hadamard matrices and certain LSSDs with Menon parameters. We then construct examples of non-trivial LSSDs in which $w$ can be made arbitarily large for fixed even part of $v$. Finally we survey the known infinite families of symmetric designs and show, using basic number theoretic conditions, that $w=2$ in most cases. 
\end{abstract}
\section{Introduction}
In \cite{Cameron}, Cameron investigated groups with inequivalent doubly-transitive permutation representations having the same permutation character and introduced the notion of a linked system of symmetric designs (LSSD). One such example arising from Kerdock codes was communicated by Goethals and later published in \cite{GOETHALS197643}. These structures in the homogeneous case were then further studied by Noda \cite{Noda} where he bounded the number of fibers in a LSSD in terms of the design parameters induced between any two of the fibers. Using the specific $(16,6,2)$ designs, Mathon \cite{Mathon} classified all inequivalent LSSDs on these parameters via a computer search, finding that there were many inequivalent LSSDs on two or three fibers but only the scheme described by Geothals worked with four or more fibers. Later, Van Dam proved in \cite{VanDam} the equivalence between these objects and 3-class Q-antipodal association schemes. Martin, Muzychuk, and Williford found a connection to mutually unbiased bases in certain dimensions \cite{MMW}. Finally Davis, Martin, Polhill \cite{DMP} and Jedwab, Li, Simon \cite{LDS} built more non-trivial examples using difference sets in 2-groups.\newpar
We begin with a survey of known results focusing on the connection to association schemes. We then introduce ``linked simplices" and establish their equivalence to LSSDs. We compare three known bounds on the number of fibers and then explore connections to structures in Euclidean space. We show how to construct Equiangular lines from arbitrary LSSDs and explore cases where LSSDs lead to real mutually unbiased bases (MUBs). After reviewing known examples, we focus on the case of Menon parameters and, employing an equivalence with sets of mutually unbiased Hadamard matrices, we construct new families of LSSDs for many values of $v$. In an appendix, we survey the parameters of known infinite families of symmetric designs and determine which of these cannot produce LSSDs on more than two fibers.
\section{Homogeneous linked systems of symmetric designs}
We begin by reviewing symmetric designs as these will play a central role in all that follows. A \textit{symmetric 2-design} with parameters $(v,k,\lambda)$ is a set of blocks $\mathcal{B}$ on point set $X$ written as $(X,\mathcal{B})$ satisfying the following three conditions:
\begin{itemize}
	\item There are $v$ blocks and $v$ points ($\vert\mathcal{B}\vert = \vert X\vert = v$);
	\item Every block contains $k$ points and every point is contained in $k$ blocks;
	\item Every pair of points is contained in $\lambda$ blocks and the intersection of any pair of blocks contains $\lambda$ points.
\end{itemize}
We form an \textit{incidence} matrix $B$ for the block design, indexing rows by blocks and columns by points, setting $B_{ij} = 1$ if point $j$ is in block $i$ and $B_{ij} = 0$ otherwise. Finally, we note the following two equivalent equations which hold for any symmetric 2-design:
\begin{align}
	k(k-1) &= \lambda(v-1)\label{sym:1}\\
	k(v-k) &= (k-\lambda)(v-1).\label{sym:2}
\end{align}
We now move to a description of a homogeneous\footnote{Here, ``homogeneous" refers to the designs between fibers all having the same parameters. For the duration of this paper, we will only concern ourselves with this case, though we drop this clarification later and only refer to the structures as linked systems of symmetric designs.} linked system of symmetric designs as described by Cameron in \cite{Cameron} and Noda in \cite{Noda}. 
Consider a multipartite graph $\Gamma$ on $wv$ vertices with vertex set partitioned into $w$ sets of $v$ vertices called ``fibers":
\[X = X_1\dot{\cup} X_2\dot{\cup}\cdots\dot{\cup}X_w.\]
We say $\Gamma$ is a \emph{linked system of symmetric designs}, $LSSD(v,k,\lambda;w)$ ($w\geq 2$), if it satisfies the following three properties:
\begin{enumerate}[(i)]
	\item no edge of $\Gamma$ has both ends in the same fiber $X_i$;
	\item for all $1\leq i,j\leq w$ with $i\neq j$, the induced subgraph of $\Gamma$ between $X_i$ and $X_j$ is the incidence graph of some $(v,k,\lambda)$-design;
	\item there exist constants $\mu$ and $\nu$ such that for distinct $h,i,j$ ($1\leq h,i,j\leq w$), 
	\begin{equation}\label{munudef}
	a\in X_i, b\in X_j \Rightarrow\vert\Gamma(a)\cap\Gamma(b)\cap X_h\vert=\begin{cases}
	\mu & a\sim b\\
	\nu & a\not\sim b
	\end{cases}
	\end{equation}
\end{enumerate}
where $\sim$ denotes adjacency in $\Gamma$ and $\Gamma(x)$ denotes the neighborhood of vertex $x$. Observe that $\Gamma$ is regular with valency $k(w-1)$.\newpar
The values of $\mu$ and $\nu$ are constrained heavily by the geometry of our graph. Assume $w\geq 3$ and consider the induced subgraph on three fibers $X_1$, $X_2$, and $X_3$. With an appropriate ordering on the vertices, the adjacency matrix of this induced subgraph has block form given by
\[A = \left[\begin{array}{ccc}
0 & B_3 & B_2^T\\
B_3^T & 0 & B_1\\
B_2 & B_1^T & 0
\end{array}\right]\]
where $B_i$ is the incidence matrix with rows indexed by $X_{i+1}$ and columns indexed by $X_{i-1}$ (with subscripts computed modulo three). Our definition implies that
\begin{equation}\label{incidence}
B_h^TB_h = B_hB_h^T = (k-\lambda)I + \lambda J
\end{equation}
for $h=1,2,3$. Moreover, our definitions of $\mu$ and $\nu$ give us three equations of the form
\begin{equation}
B_3 B_1 = \nu J + (\mu-\nu) B_2^T.\label{B1B3}
\end{equation}
Using this equation, we multiply both sides on the right by the all ones matrix to arrive at the equation
\begin{equation}
\nu=\frac{k(k+(\mu-\nu))}{v}.\label{nu}
\end{equation}
If we instead multiply Equation \eqref{B1B3} on the left by $B_3^T$, we find:
\[\left(\nu(k-(\mu-\nu))-\lambda k\right)J + \left((\mu-\nu)^2-(k-\lambda)\right)B_1=0.\]
Since $B_1$ and $J$ are linearly independent, this means that $\nu (k-s)=\lambda k$ and $(\mu-\nu)^2 = k-\lambda$. Defining $s = \sqrt{k-\lambda}$, we use the latter condition to give us that $\mu = \nu\pm s$. As we know $\mu$ and $\nu$ are both integers, this tells us that we must have $\sqrt{k-\lambda}\in\mathbb{Z}^+$ whenever $w>2$. From here on, we assume $s\in \mathbb{Z}$. Plugging this into equation \eqref{nu}, we find that $\nu = \frac{k(k\pm s)}{v}$. We now have expressions for $\mu$ and $\nu$ given by
\begin{equation}
	\nu = \frac{k(k\pm \sqrt{k-\lambda})}{v}, \qquad \mu = \nu\mp \sqrt{k-\lambda}. \label{mu-nu}
\end{equation}

As we now have two possibilities for $\mu$ and $\nu$, it becomes useful to distinguish between the two types of LSSDs. We will refer to the LSSD as ``$\mu$-heavy'' (resp., ``$\nu$-heavy'') when $\mu>\nu$ (resp., $\nu>\mu$). Note that $\mu\neq \nu$ since $k-\lambda$ is positive. We now show that swapping adjacency between fibers produces another LSSD; we call this graph the multipartite complement of $\Gamma$.
\begin{proposition}[Noda]
	Let $\Gamma$ be a  $LSSD(v,k,\lambda;w)$ with $w>2$. If $\Gamma$ is $\mu$-heavy (resp., $\nu$-heavy), the multipartite complement $\Gamma'$ is a $\nu$-heavy (resp., $\mu$-heavy) $LSSD(v,v-k,v-2k+\lambda;w)$.
\end{proposition}
\begin{proof}
	Let $\Gamma$ be a $\mu$-heavy $LSSD(v,k,\lambda;w)$. One easily checks the subgraph of $\Gamma^\prime$ induced between any two fibers is the incidence graph of a symmetric design with the following parameters.
	\[\begin{aligned}
	v^\prime = v,\qquad k^\prime = v-k,\qquad \lambda^\prime = v-2k+\lambda.
	\end{aligned}\]
	What remains to verify is the existence of integers $\mu^\prime$ and $\nu^\prime$ satisfying condition (iii). Since $k^\prime-\lambda^\prime=k-\lambda$, the parameter $s$ is the same for both LSSDs. Now consider three distinct fibers $X_i$, $X_j$, and $X_h$ and choose vertices $a\in X_i$ and $b\in X_j$ such that $a$ is adjacent to  $b$ in $\Gamma$. Then $a$ is not adjacent to $b$ in $\Gamma^\prime$ and their number of common neighbors in $X_h$ in $\Gamma^\prime$ will be the number of vertices in $X_h$ which were adjacent to neither $a$ nor $b$ in $\Gamma$. We find from \eqref{mu-nu},
	\[\nu^\prime = v-2k+\mu=\frac{k^\prime(k^\prime+s)}{v}.\]
	We repeat the same argument for two vertices not adjacent in $\Gamma$ to find $\mu^\prime = v -2k+\nu = \nu^\prime - s$, showing $\Gamma^\prime$ is $\nu$-heavy. The other implication follows by an identical proof.
\end{proof}

At first sight, it is not clear that $\mu$ and $\nu$ are uniquely determined by $v$, $k$, and $\lambda$. We now show only one of the outcomes in \eqref{mu-nu} is possible for $v\geq 3$.
\begin{lemma}
	\label{gcd}
	Let $\Gamma$ be a  $LSSD(v,k,\lambda;w)$ with $w>2$ and $1<k<v-1$. Then the following hold:
	\begin{enumerate}[(i)]
		\item exactly one of $\frac{k(k+ s)}{v}$ and $\frac{k(k-s)}{v}$ is an integer;
		\item $\gcd(k,v)>1$;
		\item $\gcd(s,v)>1$.
	\end{enumerate}
\end{lemma}
\begin{proof}
	Assume first (taking complements if necessary) that $k\leq\frac{v}{2}$. One can quickly check that $k+s<v$ except when $(v,k,\lambda) = (2,1,0)$ and $0<k-s$ whenever $k>1$. Then, for any $x\in\mathbb{Z}^+$, we have that if $\frac{x(k\pm s)}{v}$ is an integer then $\gcd(x,v) >1$. Using our expression for $\nu$ in \eqref{mu-nu}, we see that $\gcd(k,v)>1$. Noting that $k' = v-k$ and therefore (ii) holds: $\gcd(k',v) = \gcd(k,v)>1$. Now, using \eqref{sym:1} and our requirement that $\sqrt{k-\lambda} = s$, we have the two equations:
	\[\begin{aligned}
	\frac{k(k\pm s)}{v} \mp\frac{s(k\pm s)}{v} &= \lambda.
	\end{aligned}\]
	The integer $\nu$ in \eqref{mu-nu} appears as the first term in one of these equations. Therefore we must have that either $\frac{s(k+ s)}{v}$ or $\frac{s(k-s)}{v}$ must also be an integer, giving us (iii): $\gcd(s,v)>1$. To show (i), consider that if both $\frac{k(k+ s)}{v}$ and $\frac{k(k- s)}{v}$ are integers, then the same must hold for both $\frac{s(k+s)}{v}$ and $\frac{s(k-s)}{v}$. However, this implies $\frac{2s^2}{v}\in\mathbb{Z}^+$, contradicting $s^2 <k \leq \frac{v}{2}$.
\end{proof}
The case where $k=1$ or $k=v-1$ produce LSSDs which are not of interest to us and for the remainder of the paper, we will refer to these designs as degenerate. For a further description of why these designs are degenerate, see Section \ref{degenerate}. The observations that $\gcd(k,v)>1$ and $\gcd(s,v)>1$ are two tools which help us determine more easily which parameters might be feasible for a LSSD with $w>2$. There are many other statements similar to these we can find, but these two will be sufficient for now. Using these, we can immediately rule out many parameters, for instance:
\begin{corollary}
	Assume $w>2$. If there exists a non-degenerate $LSSD(v,k,\lambda;w)$, then $v$ is composite.\qed
\end{corollary}
For a further use of these tools to rule out certain families of symmetric designs, see Section \ref{families}.

\section{The association scheme structure}
Let $X$ be a finite set of vertices. A \textit{symmetric d-class association scheme} (see \cite{BCN}) on $X$ is a pair $\mathcal{L} = (X,\mathcal{R})$ where $\mathcal{R} =\left\{R_0,R_1,\dots,R_d\right\}$ is a set of $d+1$ relations on $X$ satisfying the following properties:
\begin{itemize}
	\item $R_0$ is the identity relation;
	\item $\left\{R_0,R_1,\dots, R_d\right\}$ forms a partition of $X\times X$;
	\item $(x,y)\in R_i$ implies $(y,x)\in R_i$;
	\item for $0\leq i,j,k\leq d$ there exist \textit{intersection numbers} $p_{i,j}^k$ such that for any $(x,y)\in R_k$, the number of vertices $z$ for which $(x,z)\in R_i$ and $(z,y)\in R_j$ is equal to $p_{i,j}^k$ independent of our original choice of $x$ and $y$.
\end{itemize}
Often it becomes useful to order the vertices in $X$ and then represent each $R_i$ as a 01-matrix $A_i$ where the $(x,y)$ entry of $A_i$ is 1 if and only if $(x,y)\in R_i$. With this setting in mind, the defining properties above are encoded as:
\begin{itemize}
	\item $A_0 = I$;
	\item $\sum_i A_i = J$;
	\item for all $0\leq i\leq d$, $A_i^T = A_i$;
	\item for all $0\leq i,j,k\leq d$, $A_iA_j = \sum p_{i,j}^k A_k$.
\end{itemize}
The final condition tells us that $\mathbb{A} = \text{span}\left\{A_0,A_1,\dots A_d\right\}$ forms a matrix algebra under standard matrix multiplication. As our matrices are 01-matrices with disjoint support, this \emph{Bose-Mesner algebra} is also closed under Schur (element-wise) products. Using our symmetric property, we note that $p_{i,j}^k = p_{j,i}^k$ telling us that $A_iA_j = A_jA_i$ and our matrices commute with each other. This allows us to simultaneously diagonalize our matrices to give us $d+1$ orthogonal eigenspaces with projection operators $E_0,\dots,E_d$. As both $\left\{A_0,\dots,A_d\right\}$ and $\left\{E_0,\dots,E_d\right\}$ form bases for the Bose-Mesner algebra, there exists unique matrices $P$ and $Q$ so that
\begin{equation}
\label{PQmat}
	A_i = \sum_{j} P_{ji} E_j,\qquad E_j = \frac{1}{\vert X\vert} \sum_{i} Q_{ij}A_i.
\end{equation}
We call $P$ and $Q$ the first and second eigenmatrices, respectively and note here that $P_{0i}$ is the valency of relation $R_i$ and $Q_{0j}$ is the rank of $E_j$. Finally, as our matrix algebra is closed under Schur products, we find that there exist structure constants $q_{i,j}^k$ such that for all $0\leq i,j,k\leq d$:
\[E_i\circ E_j = \frac{1}{\vert X\vert}\sum_k q_{i,j}^k E_k.\]
We call these parameters the Krein parameters of the association scheme. A \textit{$Q$-polynomial} (\textit{cometric}) association scheme is one in which the set $\left\{E_0,E_1,\dots,E_d\right\}$ may be ordered so that $q^{k}_{i,j} = 0$ whenever $k>i+j$ or $k<\vert i- j\vert$ and $q^{k}_{i,j}>0$ whenever $k = i+j$. Finally, we say an association scheme with $Q-$polynomial ordering $E_0,\dots,E_d$ is \textit{$Q$-antipodal} if $q^{k}_{d,d} >0$ when $k = 0$ or $d$ but $q^{k}_{d,d} = 0$ otherwise. Given a $Q$-polynomial ordering $E_0,\dots,E_d$ we find it convenient to order relations so that $Q_{01}>Q_{11}>\dots>Q_{d1}$; we call this the natural ordering.
\begin{theorem}[Van Dam \cite{VanDam}]
	\label{Qpoly}
  	Let $\Gamma$ be a non-degenerate LSSD with adjacency matrix $A$. Then the algebra $\langle A\rangle$ is the Bose-Mesner algebra of a 3-class Q-antipodal association scheme on $X$. Conversely, every $Q$-antipodal 3-class association scheme arises in this way. More specifically, the natural ordering of the relations of any $Q$-antipodal 3-class association scheme is as follows:
	\begin{itemize}
		\item $R_0$ is the identity relation on $X$;
		\item $R_1$ is given by adjacency in a $\mu$-heavy LSSD;
		\item $R_2$ is the union of complete graphs on the fibers induced by $R_1$;
		\item $R_3$ is given by adjacency in a $\nu$-heavy LSSD;
	\end{itemize}
\end{theorem}
\begin{proof}
	We give a brief proof of the first direction; for a full proof, see \cite{VanDam}. Let $\Gamma_1 = (X,R_1)$ be a $\mu$-heavy $LSSD(v,k,\lambda)$ with complement design given by $\Gamma_2 = (X,R_3)$. Many of the properties of an association scheme follow trivially from our relations, though the intersection numbers take more work to verify. We claim the following values for each intersection number, listed via the four matrices $L_0,L_1,L_2,L_3$ where $L_i = [p_{ij}^k]_{k,j}$. Note that the $j^{th}$ column of $L_i$ is equal to the $i^{th}$ column of $L_j$ and thus there are only 10 distinct columns in these 4 matrices.
\[\begin{aligned}
L_0 &= \left[\begin{array}{cccc}
1 & 0 & 0 & 0\\
0 & 1 & 0 & 0\\
0 & 0 & 1 & 0\\
0 & 0 & 0 & 1\\
\end{array}\right], & L_1 = \left[\begin{array}{cccc}
0 & k(w-1)       & 0   & 0               \\
1 & \mu(w-2)     & k-1 & (k-\mu)(w-2)    \\
0 & \lambda(w-1) & 0   & (k-\lambda)(w-1)\\
0 & \nu(w-2)     & k   & (k-\nu)(w-2)    \\
\end{array}\right],\\
L_2 &= \left[\begin{array}{cccc}
0 & 0   & v-1 & 0    \\
0 & k-1 & 0   & v-k  \\
1 & 0   & v-2 & 0    \\
0 & k   & 0   & v-k-1\\
\end{array}\right],
& L_3 = \left[\begin{array}{cccc}
    0 & 0             & 0     & (v-k)(w-1)\\
0 & (k-\mu)(w-2)  & v-k   & (v+\mu-2k)(w-2)\\
0 & (k-\lambda)(w-1) & 0     &(v+\lambda-2k)(w-1)\\
1 & (k-\nu)(w-2)  & v-k-1 & (v+\nu-2k)(w-2)\\
\end{array}\right].\\
\end{aligned}\]
The fact that the entries of $L_0$ are well-defined is trivial; the remaining columns of $L_1$ are well-defined by our definition of a LSSD. The values in the columns of $L_3$ follow via considering the complement of $\Gamma_1$ and $L_2$ can be computed by noting that the row sums of each matrix must be constant. 
\newpar
What remains to prove is that this association scheme is $Q$-polynomial. To show this, we must find the Krein parameters. We calculate the first and second eigenmatrices ($P$ and $Q$ respectively) given below:
\begin{equation}\label{PandQ}
P = \left[\begin{array}{cccc}
1&k(w-1)&v-1&(v-k)(w-1)\\
1&\sqrt{k-\lambda}(w-1)&-1&-\sqrt{k-\lambda}(w-1)\\
1&-\sqrt{k-\lambda}&-1&\sqrt{k-\lambda}\\
1&-k&v-1&k-v\\
\end{array}\right],\qquad Q = \left[\begin{array}{cccc}
1 & v-1 & (w-1)(v-1) & w-1\\
1 & \frac{v-k}{\sqrt{k-\lambda}} & -\frac{v-k}{\sqrt{k-\lambda}} & -1\\
1 & -1 & 1-w & w-1\\
1 & \frac{-k}{\sqrt{k-\lambda}} & \frac{k}{\sqrt{k-\lambda}} & -1\\
  \end{array}\right].
  \end{equation}
We can now use this matrix $Q$ to calculate our Krein parameters using standard techniques (see \cite{BCN}). Defining $L_i^* = [q_{i,j}^k]_{k,j}$ similar to before, it is sufficient to check that $L_1^*$ is irreducible tridiagonal (\cite[Prop.~2.7.1(i')]{BCN}) which is given below:
\[L_1^* = \left[\begin{array}{cccc}
0 & v-1 & 0 & 0\\
1 & \frac{(1-w)(2k-v)+(v-2)s}{ws} &\frac{(w-1)\left(s(v-2)+(2k-v)\right)}{ws}& 0\\
0&\frac{s(v-2)+2k-v}{ws} & \frac{s(w-1)(v-2)-(2k-v)}{ws}& 1\\
0& 0 & v-1& 0\\
\end{array}\right].\]
In order to guarantee this matrix is irreducible, we must have $s(v-2)>v-2k$. This will always hold so long as $k>1$ as we assumed in the theorem. Finally, we consider $L_3^*$ shown below:
\[L_3^* = \left[\begin{array}{cccc}
0&0 &0 & w-1\\
0& 0& w-1& 0\\
0& 1& w-2& 0\\
1& 0& 0& w-2\\
\end{array}\right]\]
and note that the final column tells us that $q^1_{d,d} = q^{2}_{d,d} = 0$ giving us that this association scheme is $Q$-antipodal.
\end{proof}
We henceforth use the term ``linked system of symmetric designs" to refer to either the graph $\Gamma$ or to the association scheme it generates as in Theorem \ref{Qpoly}. It is worth noting that we can fulfill the LSSD conditions using degenerate parameters. In this case however, we do not satisfy the requirement $s(v-2)>v-2k$ (noting that $k=1$ as our $k$ comes from the $\mu$-heavy LSSD). Therefore, while it is quite easy to build an association scheme using these parameter sets, that association scheme will not be $Q$-polynomial. This is one reason why we will ignore this case for much of our discussion.
\section{Linked Simplices}
\label{linked simplices}
In this section, we will write $\left\{b_j\right\}$ for the set $\left\{b_1,\dots,b_v\right\}$ for both sets of points and sets of blocks. For our purposes a ``regular simplex" will be taken to be a set of $v$ unit vectors spanning $\mathbb{R}^{v-1}$ with the property that the inner product of any pair of distinct vectors is $\frac{-1}{v-1}$. Let $\mathcal{A} = \left\{a_i\right\}$ and $\mathcal{B} = \left\{b_j\right\}$ be two regular simplices in $\mathbb{R}^{v-1}$. We call $\mathcal{A}$ and $\mathcal{B}$ ``linked simplices" if there exist two real numbers $\gamma$ and $\zeta$ such that for all $0\leq i,j\leq v$, we find $\left<a_i,b_j\right>\in \left\{\gamma,\zeta\right\}$. The next few theorems establish the equivalence of collections of $w$ linked simplices with LSSDs on $w$ fibers.
\begin{theorem}
\label{Association2Simplex}
Consider a $LSSD(v,k,\lambda;w)$ with Bose-Mesner algebra $\mathbb{A}$. The first idempotent $E_1$ in a $Q$-pollynomial ordering of $\mathbb{A}$, appropriately scaled, is the Gram matrix of a set of $w$ linked simplices. In the case $w=2$, $E_2$ scaled similarly is also the Gram matrix of a second set of two linked simplices.
\end{theorem}
\begin{proof}
Let $(X,\mathcal{R})$ be a $LSSD(v,k,\lambda;w)$ with Bose-Mesner algebra $\mathbb{A}$. Let $\left\{A_i\right\}$ and $\left\{E_j\right\}$ be the bases of Schur and matrix idempotents respectively. We have from \eqref{PQmat}
\[E_j = \frac{1}{\vert X\vert}\sum Q_{ij}A_i.\]
As $E_j$ is an idempotent, $E_j$ is a positive semidefinite (p.s.d.) matrix with rank $Q_{0j}$. Therefore using \eqref{PandQ},
\[G = \frac{vw}{v-1}E_1 = A_0 + \frac{v-k}{(v-1)\sqrt{k-\lambda}}A_1 -\frac{1}{v-1}A_2 -\frac{k}{(v-1)\sqrt{k-\lambda}}A_3\]
is p.s.d. with 1 on the main diagonal. Given that $Q_{01} = v-1$, $G$ is the Gram matrix of a set $Y$ of $vw$ vectors in $\mathbb{R}^{v-1}$. Further there are only three possible inner products among distinct vectors of $Y$ given by:
\[\alpha_1 = \frac{v-k}{(v-1)\sqrt{k-\lambda}}\qquad\alpha_2 =-\frac{1}{v-1}\qquad\alpha_3 = -\frac{k}{(v-1)\sqrt{k-\lambda}}.\]
Since $A_2$ corresponds to complete graphs within fibers, our vectors form a set of $w$ linked regular simplices in $\mathbb{R}^{v-1}$ with $\gamma = \alpha_1$ and $\zeta = \alpha_3$ as inner products between simplices.\newpar
Similarly we have
\[\begin{aligned}
G' = \frac{vw}{v-1}E_2 &= \left(A_0 + \frac{-k}{(v-1)\sqrt{k-\lambda}}A_1 -\frac{1}{v-1}A_2 +\frac{v-k}{(v-1)\sqrt{k-\lambda}}A_3\right).\\
\end{aligned}\]
Therefore $G'$ is also the Gram matrix of a set of vectors coming from $w$ distinct simplices. However, the rank of $E_2$ is $(w-1)(v-1)$ and therefore these simplices are linked only when $w=2$ (we require that each simplex is full-dimensional in the space spanned by vectors in both simplices in order for the simplices to be considered linked). Therefore any pair of fibers from our LSSD will give us another set of linked simplices with inner products $-\alpha_1$ and $-\alpha_3$. This corresponds to choosing one of the two simplices and replacing each $x$ in that simplex by $-x$.
\end{proof}
This tells us that every LSSD gives rise to a set of linked simplices. Before proving the converse, we first prove a lemma arising from the observation that a regular simplex is an equiangular tight frame \cite{Jasper}.
\begin{lemma}
	\label{simplextobasis}
	Let $\left\{a_i\right\}$ be a regular simplex in $\mathbb{R}^{v-1}$ and let $x,y\in \mathbb{R}^{v-1}$. Then
	\[\begin{aligned}
	\sum_i \left<a_i,x\right>\left<a_i,y\right> = \frac{v}{v-1}\left<x,y\right>.
	\end{aligned}\]
\end{lemma}
\begin{proof}
	For a vector $x$, let $x(i)$ denote the $i^\text{th}$ element of $x$. For each $1\leq i\leq v$, define $\alpha_i\in \mathbb{R}^{v}$ as the unit vector
	\[\alpha_i = \sqrt{\frac{v-1}{v}}\left[a_i(1),a_i(2),\dots,a_i(v),\frac{1}{\sqrt{v-1}}\right].\]
	For $i\neq i^\prime$,
	\[\left<\alpha_i,\alpha_{i^\prime}\right> = \frac{v-1}{v}\left(\left<a_i,a_{i^\prime}\right> + \frac{1}{v-1}\right)=0.\]
	Therefore $\left\{\alpha_i\right\}$ forms an orthonormal basis for $\mathbb{R}^v$. Now define $\chi,\psi\in\mathbb{R}^v$ as:
	\[\chi = \left[x(1),x(2),\dots,x(v),0\right],\qquad\psi = \left[y(1),y(2),\dots,y(v),0\right].\]
	Then for each $i$,
	\[\begin{aligned}
	\left<\alpha_i, \chi\right> = \sqrt{\frac{v-1}{v}}\left<a_i,x\right>,\qquad\left<\alpha_i, \psi\right> = \sqrt{\frac{v-1}{v}}\left<a_i,y\right>
	\end{aligned}\]
	giving us
	\[\left<x,y\right> =\left<\chi,\psi\right>=\sum_i\left<\alpha_i, \chi\right>\left<\alpha_i, \psi\right> =\frac{v-1}{v}\sum_i\left<a_i,x\right>\left<a_i,y\right>.\]
\end{proof}
Using this lemma, we now seek to build a LSSD on $w$ fibers from a set of $w$ linked simplices. We first provide a construction of the graph $\Gamma$ and then split the verification into two parts: first that $\Gamma$ restricted to a pair of fibers represents a symmetric design, and second that the constants $\mu$ and $\nu$ given by \eqref{munudef} are well-defined. Clearly, we need only consider three fibers in the proofs to follow; the arguments extend to $w$ fibers.
\begin{theorem}
	\label{2design}
  	Let $\left\{a_i\right\}$ and $\left\{b_j\right\}$ be linked simplices in $\mathbb{R}^{v-1}$ with inner products $\gamma$ and $\zeta$. For each $j$, let $B_j = \left\{a_i:\left<a_i,b_j\right> = \gamma\right\}$. Then $(\left\{a_i\right\},\left\{B_j\right\})$ is a symmetric 2-design.
\end{theorem}
\begin{proof}
	First we must prove that each block contains a constant number of points. Let $1\leq j\leq v$ be fixed and define $k_j = \vert B_j\vert$. Since the set $\left\{a_i\right\}$ of vectors form a regular simplex, the centroid of those vectors must be the origin. Then, $\sum_i\left<b_j,a_i\right> = \left<b_j,\sum_ia_i\right> = 0$ giving us the equation $k_j\gamma + (v-k_j)\zeta = 0$. Solving this for $k_j$ gives $k_j = \frac{\zeta v}{\gamma-\zeta}$, independent of $j$. Now we will show that any pair of blocks contains a constant number of points in common; swapping roles this gives that any pair of points is contained in a constant number of blocks. Fix $1\leq s,t\leq v$ so that $b_s$ and $b_t$ are two distinct vectors in $\left\{b_j\right\}$ with corresponding blocks $B_s$ and $B_t$ respectively. Define $\lambda_{s,t} = \vert B_s\cap B_t\vert$ and
  \[x_s = \left[\left<a_1,b_s\right>,\left<a_2,b_s\right>,\dots,\left<a_v,b_s\right>\right],\qquad x_t = \left[\left<a_1,b_t\right>,\left<a_2,b_t\right>,\dots,\left<a_v,b_t\right>\right].\]
  Recalling that $k\gamma+(v-k)\zeta = 0$,
  \[\left<x_s,x_t\right> = \lambda_{s,t} \gamma^2 + 2(k-\lambda_{s,t})\gamma\zeta + (v-2k+\lambda_{s,t})\zeta^2=\lambda_{s,t}\left(\zeta-\gamma\right)^2 - v\zeta^2.\]
  We may instead apply Lemma \ref{simplextobasis} to get:
  \[\left<x_s,x_t\right> = \sum_i \left<a_i,b_s\right>\left<a_i,b_t\right>=\frac{v}{v-1}\left<b_s,b_t\right>=-\frac{v}{(v-1)^2}.\]
  Equating these two values gives us
  \[\lambda_{s,t} =\frac{v\zeta^2}{(\zeta-\gamma)^2}-\frac{v}{(v-1)^2(\gamma-\zeta)^2}.\]
  The quantity on the right is independent of $s$ and $t$ and therefore $\lambda_{s,t}$ does not depend on $s$ and $t$. Therefore our block system forms a $2$-design with the above values for $k$ and $\lambda$.
\end{proof}
As both $k$ and $\lambda$ are integers, this gives us restrictions on which inner products are allowed. We solve the system
\[\begin{aligned}
k\gamma + (v-k)\zeta &= 0,\\
\lambda(\zeta-\gamma)^2 - v\zeta^2 &= -\frac{v}{(v-1)^2}\\
\end{aligned}\]
to find that $\zeta^2 =\frac{k}{(v-1)(v-k)}$. Using \eqref{sym:2}, this simplifies to
\begin{equation}
\zeta = \pm\frac{k}{(v-1)\sqrt{k-\lambda}},\qquad\gamma = \mp\frac{v-k}{(v-1)\sqrt{k-\lambda}}.
\end{equation}
These match the previously determined entries of $E_1$ and $E_2$ corresponding to the first and third relations. Our next theorem concerns the existence of $\mu$ and $\nu$, which arise between triples of fibers.
\begin{theorem}
  \label{3fibers}
	Let $\left\{a_i\right\}$, $\left\{b_i\right\}$, and $\left\{c_i\right\}$ be linked regular simplices in $\mathbb{R}^{v-1}$ with inner products $\gamma$ and $\zeta$ as before. For each $1\leq j,k\leq v$, let $B_j = \left\{a_i:\left<a_i,b_j\right> = \gamma\right\}$ and $C_k = \left\{a_i:\left<a_i,c_k\right> = \gamma\right\}$. Then there exists integers $\mu$ and $\nu$ such that 
	\[\vert B_j\cap C_k\vert = \begin{cases}
	\mu  & \left<b_j,c_k\right> = \gamma\\
	\nu  & \left<b_j,c_k\right> = \zeta\\
	\end{cases}\]
	where $\mu$ and $\nu$ are independent of our choice of $j$ and $k$.
\end{theorem}
\begin{proof}
	We follow a similar method of calculating an inner product in two ways, then equating the results. Fix $0\leq i,j\leq v$ and let $\eta_{i,j} = \vert B_i\cap C_j\vert$. Define
	\[x_i = \left[\left<a_1,b_i\right>,\left<a_2,b_i\right>,\dots,\left<a_v,b_i\right>\right],\qquad x_j = \left[\left<a_1,c_j\right>,\left<a_2,c_j\right>,\dots,\left<a_v,c_j\right>\right].\]
	Then we have $\left<x_i,x_j\right> =\eta_{i,j}(\gamma-\zeta)^2 - v\zeta^2$ and using Lemma \ref{simplextobasis},
	\[\left<x_i,x_j\right> = \sum_\ell \left<a_\ell,b_i\right>\left<a_\ell,c_j\right>=\frac{v}{v-1}\left<b_i,c_j\right>.\]
	Equating these two values and solving for $\eta_{i,j}$ gives us:
	\[\eta_{i,j} = \frac{1}{(\gamma-\zeta)^2}\left(v\zeta^2 + \frac{v}{v-1}\left<b_i,c_j\right>\right).\]
	While the right side is not independent of $i$ and $j$ as we saw in the previous theorem, it is only dependent on the value of $\left<b_i,c_j\right>$. Using $\nu$ and $\mu$ for $\eta_{i,j}$ when $\left<b_i,c_j\right>$ is $\zeta$ and $\gamma$ respectively, we have:
	\[\begin{aligned}
	 \nu&=\frac{v}{(\gamma-\zeta)^2}\left(\frac{\zeta^2(v-1)^2+\zeta(v-1)}{(v-1)^2}\right),\\
	\mu&=\frac{v}{(\gamma-\zeta)^2}\left(\frac{\zeta^2(v-1)^2+\gamma(v-1)}{(v-1)^2}\right)= \nu + \frac{v}{(\gamma-\zeta)(v-1)}.
	\end{aligned}\]
	Using the values of $\gamma$ and $\zeta$ found previously to make $\lambda$ integral, we find that
	\[\begin{aligned}
	\gamma-\zeta&=\mp\frac{v}{(v-1)\sqrt{k-\lambda}},
	\end{aligned}\]
	giving us that
	\[\begin{aligned}
	\nu&= \frac{k(k\pm \sqrt{k-\lambda})}{v},\\
	\mu&=\nu\mp\sqrt{k-\lambda}.
	\end{aligned}\]
	Since $\mu$ and $\nu$ are both cardinalities of sets, anytime we find non-integral values for $\mu$ and $\nu$ we can conclude that the original set of linked simplices could not exist.
\end{proof}
We finish this section with the following result
\begin{theorem}
  A $LSSD(v,k,\lambda;w)$ is equivalent to a set of $w$ linked simplices in $\mathbb{R}^{v-1}$.
\end{theorem}
\begin{proof}
  Theorem \ref{Association2Simplex} tells us that given any $LSSD(v,k,\lambda;w)$ we can always build a set of $w$ linked simplices using a scaled version of the first idempotent as the Gram matrix. For the converse, let $\left\{X_1,X_2,\dots,X_w\right\}$ be a set of $w$ regular simplices with inner products $\gamma>\zeta$. Define a graph $\Gamma$ on vertex set $\bigcup_i X_i$ where $x\in X_j$ and $y\in X_\ell$ ($j\neq \ell$) are adjacent if and only if $\left<x,y\right> = \gamma$. Then $\Gamma$ is a multipartite graph on $w$ fibers. Theorem \ref{2design} tells us that the induced graph between a pair of fibers is a symmetric $2$-design. Theorem $\ref{3fibers}$ shows that given any pair of vertices in distinct fibers $x\in X_i$ and $y\in X_\ell$,
  \[\vert\Gamma(x)\cap\Gamma(y)\cap X_j\vert = \begin{cases}
    \mu & x\sim y\\
    \nu & x\not\sim y
  \end{cases}\]
  where $X_j$ is a third fiber. As we assumed $\gamma>\zeta$, this also provides that $\mu>\nu$. Therefore $\Gamma$ is a $\mu$-heavy LSSD and adjacency in $\Gamma$ is the first relation of our proposed association scheme. The third relation (the $\nu$-heavy LSSD) is built from using $\zeta$ to define adjacency.
\end{proof}
\subsection{A geometric classification}
As a collection of linked simplices, and likewise the association scheme arising from it, incorporate both the $\mu$-heavy and $\nu$-heavy LSSDs, it becomes more useful to differentiate between LSSDs with $P_{01}>P_{03}$ or vice versa. As $P_{01}$ is the valency of the $\mu$-heavy LSSD and therefore designates the number of positive inner products a given vector has with other fibers, we classify a LSSD as an ``optimistic LSSD" if $P_{01}>P_{03}$. Likewise we classify the opposite case as a ``pessimistic LSSD". At the level of linked systems, an LSSD is optimistic if $(2k-v)(\mu-\nu)>0$ and pessimistic if $(2k-v)(\mu-\nu)<0$. Consider the following table of possible LSSDs:
\[\begin{tabular}{c|c|c}
& $2k>v$ & $2k<v$\\\hline
$\mu$-heavy & optimistic LSSD & pessimistic LSSD\\\hline
$\nu$-heavy & pessimistic LSSD & optimistic LSSD\\
\end{tabular}\]
Motivated by the natural ordering of relations, we will adopt the convention of focusing on the $\mu$-heavy LSSD. This forces us to allow for $k>\frac{v}{2}$. In fact, this is the case for all known examples of LSSDs.
\section{Bounds for $w$}
\subsection*{\sc Noda Bound}
In Theorem 2 of \cite{Noda}, Noda gives the following bound:
\begin{equation*}\resizebox{\textwidth}{!}{
	$(w-1)\left[(k-2)\lambda\binom{k}{3}-(v-2)\left[(v-k)\binom{\nu}{3} + k\binom{\mu}{3}\right]\right]\leq(v-2)\left[(v-1)\binom{\lambda}{3}+\binom{k}{3}-\left[(v-k)\binom{\nu}{3} + k\binom{\mu}{3}\right]\right]$
}\end{equation*}
with equality if and only if a pair $(X_1,X_2\cup X_3\cup\dots\cup X_f)$ forms a 3-design. If we restrict ourselves to the case of $\mu$-heavy LSSDs, one can verify that this gives:
\[(w-1)(2k-v)\leq(v-2)\sqrt{k-\lambda}.\]
If we have that $2k>v$, then we arrive at the bound:
\[w \leq \frac{(v-2)\sqrt{k-\lambda}}{2k-v}+1.\]
Since the condition becomes vacuous when $2k<v$, the bound only applies to optimistic LSSDs.
\subsection*{\sc Krein Conditions}
In \cite{Mathon}, Mathon shows that the previous bound is equivalent to the requirement that $q_{1,1}^1\geq 0$.
\[\begin{aligned}
q_{1,1}^1 =\frac{(1-w)(2k-v)+(v-2)s}{ws}\geq 0,\qquad w\leq \frac{(v-2)s}{2k-v}+1
\end{aligned}\]
assuming $v<2k$. As before, the bound says nothing about pessimistic LSSDs.
\subsection*{\sc Absolute Bound}
In \cite{MMW}, a bound is provided independent of the sign of $(2k-v)(\mu-\nu)$ relying only on the $Q$-polynomial structure (particularly that $q_{1,1}^3 = 0$ and $q_{1,1}^2 >0$). Let $m_i = \text{rank}(E_i)$. We know from our $Q$ matrix that $m_2 = (w-1)m_1$. Further,
\[E_1\circ E_1 = q_{1,1}^0E_0 + q_{1,1}^1E_1+q_{1,1}^2E_2.\]
Therefore $m_2+m_1 + 1\leq \frac{1}{2}m_1\left(m_1+1\right)$ if $q_{1,1}^1>0$ and $m_2+1\leq \frac{1}{2}m_1\left(m_1+1\right)$ otherwise (c.f. \cite{BCN} Thm. 2.3.3). First consider the case when $q_{1,1}^1>0$:
\[\begin{aligned}
(w-1)m &\leq \frac{m^2-m-2}{2}\\
w&\leq \frac{m}{2}+\frac{1}{2}-\frac{2}{m}.
\end{aligned}\]
Since $m = v-1$ is an integer, this gives $w \leq \frac{v-1}{2}$.\newpar
If instead we have that $q_{1,1}^1 = 0$ then we find $w\leq \frac{v+1}{2}$, so the absolute bound is never tight. But here, know our Krein condition is tight giving us
\[w = \frac{(v-2)s}{2k-v}+1\]
with $2k>v$ ($q_{1,1}^1 = 0$ is not possible unless $2k>v$ since $w\geq 2$). Further, our bound from the eigenspace structure gives
\[w\leq \frac{m+2}{2} = \frac{v+1}{2}.\]
Therefore we must have $\frac{(v-2)s}{2k-v}+1\leq \frac{v+1}{2}$ giving $2(v-2)s\leq(v-1)(2k-v)$ or $2s\leq (2k-v) + \frac{2s}{v-1}$. Since $2s<k<v-1$, this means that $2s\leq (2k-v)$. We can square both sides and preserve the inequality, giving:
 $4(k-\lambda)\leq 4k^2-4kv+v^2$. Using \eqref{sym:2} this gives: 
\[\begin{aligned}
4(k-\lambda)\leq v.
\end{aligned}\]
This means that whenever the Krein condition is tight ($q_{1,1}^1 = 0$), we must have $4(k-\lambda)\leq v\leq 2k$. There is only one known family of constructions which achieve the Krein bound. For this construction (see \ref{kerdock}) $v = 4(k-\lambda)$. Further, the more general family of parameters (Menon) is the only possible family for which this inequality is tight.

\section{Connections to Structures in Euclidean Space}
In section \ref{linked simplices}, we explored the equivalence of LSSDs with a certain geometric structure using the first and second idempotents as Gram matrices. We now explore similar structures which can be built using combinations of these idempotents, though we will not always be able to reverse these constructions. Recall that,
\[E_j = \frac{1}{\vert X\vert}\sum Q_{ij}A_i\]
and the rank of $E_j$ is given by $Q_{0j}$. By considering nonnegative linear combinations of these idempotents, we construct Gram matrices of systems of vectors with desirable properties. As we are interested in low rank Gram matrices, we will only consider combinations of two or three of these idempotents, avoiding $E_2$ as this has rank $(w-1)(v-1)$. Before moving to the examples, we note that the matrix $x_0E_0 + x_1E_1+x_3E_3$ is expressible as $\sum_i y_i A_i$ with the following values for $y_i$:
\[\begin{aligned}
y_0 &= \frac{1}{vw}(x_0 + (v-1)x_1 + (w-1) x_3),\\
y_1 &= \frac{1}{vw}(x_0 + \frac{v-k}{\sqrt{k-\lambda}}x_1 -x_3),\\
y_2 &= \frac{1}{vw}(x_0 - x_1 +(w-1)x_3),\\
y_3 &= \frac{1}{vw}(x_0 -\frac{k}{\sqrt{k-\lambda}}x_1 - x_3).
\end{aligned}\]

\subsection{Equiangular Lines}
This construction is a generalization of de Caen's construction \cite{deCaen} using the Cameron-Seidel scheme. Using the above idempotents, we wish to make a positive semi-definite matrix with low rank with constant diagonal and only one possible magnitude of the main diagonal. Since the rank of our matrix is the sum of the multiplicities of the idempotents we use, we avoid using $E_2$ in order to reduce the rank. Therefore consider the following matrix:
\[G = vw(\alpha E_0 + \beta E_1 + \gamma E_3).\]
This matrix will be a $vw\times vw$ matrix with rank $v+w-1$. In order to get equiangular lines, we must have a constant positive value $c$ such that
\[\begin{aligned}
\Big| \alpha + \beta\left(\frac{v-k}{\sqrt{k-\lambda}}\right) - \gamma\Big| =\Big| \alpha - \beta +(w-1)\gamma\Big| =\Big| \alpha - \beta\left(\frac{k}{\sqrt{k-\lambda}}\right) - \gamma\Big| = c.
\end{aligned}\]
Since $\alpha$, $\beta$, and $\gamma$ must all be positive, we note that $\alpha+\beta\left(\frac{v-k}{\sqrt{k-\lambda}}\right) - \gamma >\alpha - \beta\left(\frac{k}{\sqrt{k-\lambda}}\right)-\gamma$ and therefore we must have
\[c=\alpha+\beta\left(\frac{v-k}{\sqrt{k-\lambda}}\right) - \gamma = -\left[\alpha-\beta\left(\frac{k}{\sqrt{k-\lambda}}\right) - \gamma\right] .\]
This tells us that
\[\beta = \frac{2c\sqrt{k-\lambda}}{v}\qquad \text{and}\qquad \alpha - \gamma = c\left(\frac{2k-v}{v}\right).\]
Here we have one final choice (the sign of $\alpha-\beta+(w-1)\gamma$). Plugging in our value for $\beta$, we find that 
\[\alpha+(w-1)\gamma = c\left(\frac{2\sqrt{k-\lambda}\pm v}{v}\right).\]
However, since we must have $2\sqrt{k-\lambda}<v$, we know that choosing the minus on the right hand side would make the entire side negative. However $\alpha$, $\gamma$, and $(w-1)$ are all positive and therefore this is not possible. Therefore we must use the $+$, giving
\[\gamma=2c\left(\frac{v-k + \sqrt{k-\lambda}}{vw}\right),\qquad
\alpha = c\left(\frac{v+2\sqrt{k-\lambda} - (w-1)\left(v-2k\right)}{vw}\right).\]
Our final requirement is that the main diagonal of $G$ is equal to 1. Using the coefficients of $A_0$ in our expression for $G$, we find
\[\begin{aligned}
c=\frac{1}{2\sqrt{k-\lambda}-1}.
\end{aligned}\]
Scaling by $vw$ for convenience, this gives us the final values
\[\begin{aligned}vw\alpha&=\frac{v+2\sqrt{k-\lambda} - (w-1)(v-2k)}{(2\sqrt{k-\lambda}-1)},\\
 vw\beta&=\frac{2w\sqrt{k-\lambda}}{2\sqrt{k-\lambda}-1},\\
 vw\gamma&=\frac{2v-2k+2\sqrt{k-\lambda}}{2\sqrt{k-\lambda}-1},\\
 \end{aligned}\]
 with inner product $\frac{1}{2\sqrt{k-\lambda} -1}$. As we require $\alpha\geq 0$ for this construction to give a p.s.d. matrix, we need $w<\frac{2(k+s)}{v-2k}+2$ in the pessimistic LSSD case (no such restriction occurs for the optimistic case). Therefore we have the following generalization of de Caen's construction:
 \begin{theorem}
 	Let $\mathcal{L}$ be the association scheme arising from a $LSSD(v,k,\lambda;w)$. If either $\mathcal{L}$ is optimistic, or $w\leq2+\frac{2(k+s)}{v-2k}$ then we can build a set of $vt$ lines in $\mathbb{R}^{v+t-1}$ for any $1\leq t\leq w$. In the pessimistic case with $w> 2+\frac{2(k+s)}{v-2k}$, we can achieve the construction for any $t$ less than or equal to this bound.\qed
 \end{theorem}
\subsection{Mutually Unbiased Bases}
\label{bases}
A second useful system of vectors arises from taking $x_2 = x_3 = 0$ and $x_0 = x_1 = w$. This gives us the Gram matrix
\[G = A_0 + \frac{v-k+\sqrt{k-\lambda}}{v\sqrt{k-\lambda}}A_1 -\frac{k-\sqrt{k-\lambda}}{\sqrt{k-\lambda}}A_3\]
of a set of $w$ orthonormal bases where two vectors from distinct bases have one of two inner products;
\[\beta_1 = \frac{v-k+\sqrt{k-\lambda}}{v\sqrt{k-\lambda}},\qquad
\beta_2= -\frac{k-\sqrt{k-\lambda}}{v\sqrt{k-\lambda}}.\]
Of particular interest is the case when $\vert \beta_1\vert = \vert\beta_2\vert;$ this is precisely when our construction gives a set of mutually unbiased bases. This will be discussed in greater detail in Section \ref{LSSDMUBs}.
\section{Examples}
\subsection{Degenerate case}
\label{degenerate}
We first examine the case when the $Q$-polynomial structure fails as discussed in the proof of Theorem $\ref{Qpoly}$. Arguably the most interesting property of this scheme is that there are no bounds on $w$. In fact, for any choice of $v,w>0$, we can build a LSSD with $w$ fibers by imposing an ordering on the vertices in each fiber and connecting vertices with the same values, giving us a $\mu$-heavy LSSD with the complement giving us the $\nu$-heavy LSSD. Below is a representation of the complementary pairs LSSD(4,1,0;3) and LSSD(4,3,2;3).
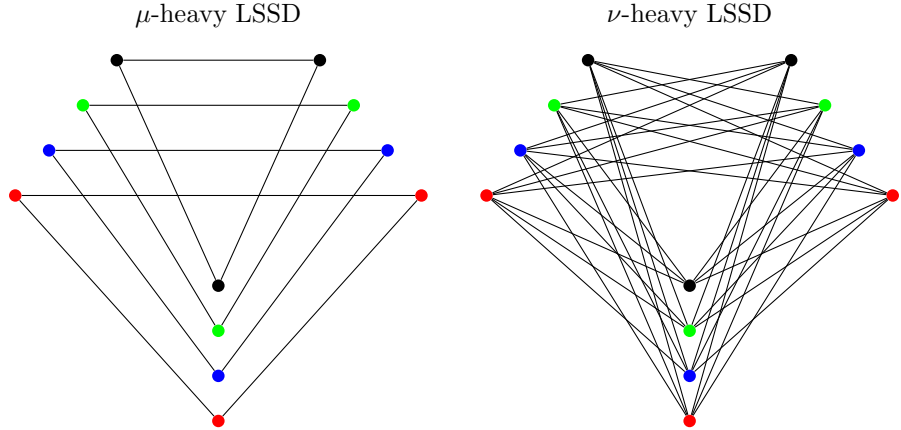
\[\begin{tikzpicture}[scale = .6,node distance=3cm,
thin,main node/.style={circle,fill=black,scale = .5}]
\node at (0,5) {$\mu$-heavy LSSD};
\def\x{1.5};
\node[main node,color = red] (11) at (-3*\x,1) {};
\node[main node,color = blue] (12) at (-2.5*\x,2) {};
\node[main node,color = green] (13) at (-2*\x,3) {};
\node[main node] (14) at (-1.5*\x,4) {};
\node[main node] (24) at (1.5*\x,4) {};
\node[main node,color = green] (23) at (2*\x,3) {};
\node[main node,color = blue] (22) at (2.5*\x,2) {};
\node[main node,color = red] (21) at (3*\x,1) {};
\node[main node] (34) at (0*\x,-1) {};
\node[main node,color = green] (33) at (0*\x,-2) {};
\node[main node,color = blue] (32) at (0*\x,-3) {};
\node[main node,color = red] (31) at (0*\x,-4) {};

\draw [-] (11) -- (21) -- (31) -- (11);
\draw [-] (12) -- (22) -- (32) -- (12);
\draw [-] (13) -- (23) -- (33) -- (13);
\draw [-] (14) -- (24) -- (34) -- (14);
\end{tikzpicture}\qquad
\begin{tikzpicture}[scale = .6,node distance=3cm,
thin,main node/.style={circle,fill=black,scale = .5}]
\node at (0,5) {$\nu$-heavy LSSD};
\def\x{1.5};
\node[main node,color = red] (11) at (-3*\x,1) {};
\node[main node,color = blue] (12) at (-2.5*\x,2) {};
\node[main node,color = green] (13) at (-2*\x,3) {};
\node[main node] (14) at (-1.5*\x,4) {};
\node[main node] (24) at (1.5*\x,4) {};
\node[main node,color = green] (23) at (2*\x,3) {};
\node[main node,color = blue] (22) at (2.5*\x,2) {};
\node[main node,color = red] (21) at (3*\x,1) {};
\node[main node] (34) at (0*\x,-1) {};
\node[main node,color = green] (33) at (0*\x,-2) {};
\node[main node,color = blue] (32) at (0*\x,-3) {};
\node[main node,color = red] (31) at (0*\x,-4) {};

\draw [-] (11) -- (22);
\draw [-] (11) -- (23);
\draw [-] (11) -- (24);
\draw [-] (11) -- (32);
\draw [-] (11) -- (33);
\draw [-] (11) -- (34);

\draw [-] (12) -- (21);
\draw [-] (12) -- (23);
\draw [-] (12) -- (24);
\draw [-] (12) -- (31);
\draw [-] (12) -- (33);
\draw [-] (12) -- (34);

\draw [-] (13) -- (32);
\draw [-] (13) -- (31);
\draw [-] (13) -- (34);
\draw [-] (13) -- (22);
\draw [-] (13) -- (21);
\draw [-] (13) -- (24);

\draw [-] (14) -- (22);
\draw [-] (14) -- (23);
\draw [-] (14) -- (21);
\draw [-] (14) -- (32);
\draw [-] (14) -- (33);
\draw [-] (14) -- (31);

\draw [-] (21) -- (32);
\draw [-] (21) -- (33);
\draw [-] (21) -- (34);
\draw [-] (22) -- (31);
\draw [-] (22) -- (33);
\draw [-] (22) -- (34);
\draw [-] (23) -- (32);
\draw [-] (23) -- (31);
\draw [-] (23) -- (34);
\draw [-] (24) -- (32);
\draw [-] (24) -- (33);
\draw [-] (24) -- (31);

\end{tikzpicture}\]
In this case the $\mu$-heavy LSSD has parameters $(v,1,0)$, so we find that $s = \sqrt{k-\lambda} = 1$, $\nu = \frac{k(k-s)}{v} = 0$, and $\mu = \nu+s = 1$. We list the $P$ matrix and $Q$ matrix and use these to further describe the LSSD:
\[P = \left[\begin{array}{cccc}
1&w-1&v-1&(v-1)(w-1)\\
1&w-1&-1&-(w-1)\\
1&-1&-1&1\\
1&-1&v-1&1-v\\
\end{array}\right],\qquad Q = \left[\begin{array}{cccc}
1 & v-1 & (w-1)(v-1) & w-1\\
1 & v-1 & -(v-1) & -1\\
1 & -1 & 1-w & w-1\\
1 & -1 & 1 & -1\\
  \end{array}\right].\]
This means that our first idempotent is given by:
\[E_1 = (v-1)I+(v-1)A_1-A_2-A_3 = (vI-J)\otimes J.\]
If we scale this appropriately to reach a Gram matrix, we find that $E_1$ corresponds to $w$ copies of the same simplex in $\mathbb{R}^v$. This can be seen as well from the fact that $Q_{01} = Q_{11}$ meaning that any simplex vector has inner product 1 with exactly one vector from each of the ``other" simplices, meaning the simplices are just copies of the same simplex. This explains why $w$ is unbounded as we can always copy the same simplex as many times as we would like. This also indicates why this example is not of interest to us as it is not giving $w$ distinct linked simplices.
\subsection{Cameron-Seidel Scheme}
\label{kerdock}
Let $Q_1,Q_2,\dots,Q_w$ be a set of $w$ quadratic forms on $\mathbb{Z}_2^n$ for which $Q_i+Q_j$ is a full rank quadratic form whenever $i\neq j$. Let $\mathcal{Q}_1,\mathcal{Q}_2,\dots,\mathcal{Q}_w$ be cosets of the Reed Muller code as defined above. For each $1\leq i\leq w$, define a set of vectors $V_i$ via
\[V_i = \left\{[q(1),q(2),\dots,q(2^n-1)]\quad\vert\quad q(2^n) = 0\right\}_{q\in\mathcal{Q}_i}.\]
Since $\mathcal{Q}_i$ is closed under complements, we know that $\vert V_i\vert = \frac{1}{2}\vert\mathcal{Q}_i\vert = 2^{n}$. Further, any pair of vectors $v,w\in\bigcup_iV_i$ come from vectors $q_v,q_w\bigcup_i\mathcal{Q}_i$ with last entry $0$, so $wt(v\oplus w) = wt(q_v\oplus q_w)$. Finally, for each $i$ construct the vector set
\[X_i = \left\{\frac{1}{\sqrt{2^n-1}}\left(2v-\mathbbm{1}\right)\vert v\in V_i\right\}.\]
We claim $\left\{X_i\right\}_{i=1..w}$ is a set of linked simplices. To verify this, fix $1\leq i<j\leq w$ and let $x_i,y_i\in X_i$ and $z_j\in X_j$ with corresponding coset vectors $q_x,q_y,$ and $q_z$ respectively. Then,
\[\begin{aligned}
\left<x_i,x_i\right> &= \frac{1}{2^n-1}\left((wt(x_i)+(-1)^2(2^{n}-1-wt(x_i))\right)=1
\end{aligned}\]
giving that every vector in $\bigcup_iX_i$ is a unit vector. Next,
\[\begin{aligned}
\left<x_i,y_i\right> &= \frac{1}{2^n-1}\left((2^n-1)-2wt(q_x\oplus q_y)\right)=-\frac{1}{2^n-1}
\end{aligned}\]
giving us that $X_i$ forms a regular simplex. Finally,
\[\begin{aligned}
\left<x_i,z_j\right> &= \frac{1}{2^n-1}\left((2^n-1)-2wt(q_x\oplus q_w)\right).
\end{aligned}\]
Since $wt(q_x\oplus q_z)\in \left\{2^{n-1}\pm2^{r-1}\right\}$ we have that
\[\left<x_i,w_j\right> = \begin{cases}
\frac{2^{r}-1}{2^{n}-1} & wt(q_x\oplus q_w) = 2^{n-1}-2^{r-1}\\
-\frac{2^r+1}{2^n-1} & wt(q_x\oplus q_w) = 2^{n-1}+2^{r-1}\\
\end{cases}\]
meaning there are two possible angles between simplices.\newpar
Therefore we can build a LSSD on $w$ fibers whenever we have $w$ quadratic forms whose pairwise sums are full rank. We represent each quadratic form as the $n\times n$ matrix giving the corresponding alternating bilinear form. Then, every matrix must differ in the first row in order for their difference to be full rank. This means $w\leq2^{n-1}$ as there are only $2^{n-1}$ possible choices for the first row. This upper bound is achievable whenever $n$ is even \cite{Quadforms}. Below we give an example when $n = 4$ where $Q_i$ is the alternating bilinear form corresponding to the $i^{\text{th}}$ quadratic form.
\[\begin{aligned}
Q_1 &= \left[\begin{array}{cccc}
0 & 0 & 0 & 0\\
0 & 0 & 0 & 0\\
0 & 0 & 0 & 0\\
0 & 0 & 0 & 0\\
\end{array}\right],\qquad Q_2 = \left[\begin{array}{cccc}
0 & 1 & 0 & 0\\
1 & 0 & 0 & 0\\
0 & 0 & 0 & 1\\
0 & 0 & 1 & 0\\
\end{array}\right],\qquad Q_3 = \left[\begin{array}{cccc}
0 & 0 & 1 & 0\\
0 & 0 & 0 & 1\\
1 & 0 & 0 & 1\\
0 & 1 & 1 & 0\\
\end{array}\right],\qquad
Q_4 = \left[\begin{array}{cccc}
0 & 1 & 1 & 0\\
1 & 0 & 1 & 1\\
1 & 1 & 0 & 0\\
0 & 1 & 0 & 0\\
\end{array}\right],
\\
Q_5 &= \left[\begin{array}{cccc}
0 & 0 & 0 & 1\\
0 & 0 & 1 & 1\\
0 & 1 & 0 & 1\\
1 & 1 & 1 & 0\\
\end{array}\right],\qquad
Q_6 = \left[\begin{array}{cccc}
0 & 1 & 0 & 1\\
1 & 0 & 1 & 0\\
0 & 1 & 0 & 0\\
1 & 0 & 0 & 0\\
\end{array}\right],\qquad Q_7 = \left[\begin{array}{cccc}
0 & 0 & 1 & 1\\
0 & 0 & 1 & 0\\
1 & 1 & 0 & 1\\
1 & 0 & 1 & 0\\
\end{array}\right],\qquad Q_8 = \left[\begin{array}{cccc}
0 & 1 & 1 & 1\\
1 & 0 & 0 & 1\\
1 & 0 & 0 & 0\\
1 & 1 & 0 & 0\\
\end{array}\right].\\
\end{aligned}\]
We can form the characteristic vectors $[Q_i(v)]_v$. Below we display $[Q_2(v)]_v$ and $[Q_8(v)]_v$:
\[\begin{aligned}
[Q_2(v)]_v &= \left[\begin{array}{cccccccccccccccc}0 & 0 & 0 & 1 & 0 & 0 & 0 & 1 & 0 & 0 & 0 & 1 & 1 & 1 & 1 & 0\end{array}\right],\\
[Q_8(v)]_v &= \left[\begin{array}{cccccccccccccccc}0 & 0 & 0 & 1 & 0 & 1 & 0 & 0 & 0 & 1 & 1 & 1 & 0 & 0 & 1 & 0\end{array}\right].\\
\end{aligned}\]
Given the characteristic vectors, we can find the cosets of $RM(1,4)$. The coset corresponding to $Q_2(v)$ is given below as the set of rows of the matrix. To improve readability, $+$ denotes a $1$, $-$ denotes a $-1$ and an empty space denotes a $0$.
\[[Q_2(v)]_v+RM(1,4) =\setlength{\arraycolsep}{.8pt}\scalefont{.4}{\left[\begin{array}{cccccccccccccccc}
 &  &  & + &  &  &  & + &  &  &  & + & + & + & + & \\
+ & + & + &  & + & + & + &  & + & + & + &  &  &  &  & +\\
 & + &  &  &  & + &  &  &  & + &  &  & + &  & + & +\\
+ &  & + & + & + &  & + & + & + &  & + & + &  & + &  & \\
+ & + &  & + & + & + &  & + & + & + &  & + &  &  & + & \\
 &  & + &  &  &  & + &  &  &  & + &  & + & + &  & +\\
+ &  &  &  & + &  &  &  & + &  &  &  &  & + & + & +\\
 & + & + & + &  & + & + & + &  & + & + & + & + &  &  & \\
+ & + & + &  &  &  &  & + & + & + & + &  & + & + & + & \\
 &  &  & + & + & + & + &  &  &  &  & + &  &  &  & +\\
+ &  & + & + &  & + &  &  & + &  & + & + & + &  & + & +\\
 & + &  &  & + &  & + & + &  & + &  &  &  & + &  & \\
 &  & + &  & + & + &  & + &  &  & + &  &  &  & + & \\
+ & + &  & + &  &  & + &  & + & + &  & + & + & + &  & +\\
 & + & + & + & + &  &  &  &  & + & + & + &  & + & + & +\\
+ &  &  &  &  & + & + & + & + &  &  &  & + &  &  & \\
+ & + & + &  & + & + & + &  &  &  &  & + & + & + & + & \\
 &  &  & + &  &  &  & + & + & + & + &  &  &  &  & +\\
+ &  & + & + & + &  & + & + &  & + &  &  & + &  & + & +\\
 & + &  &  &  & + &  &  & + &  & + & + &  & + &  & \\
 &  & + &  &  &  & + &  & + & + &  & + &  &  & + & \\
+ & + &  & + & + & + &  & + &  &  & + &  & + & + &  & +\\
 & + & + & + &  & + & + & + & + &  &  &  &  & + & + & +\\
+ &  &  &  & + &  &  &  &  & + & + & + & + &  &  & \\
 &  &  & + & + & + & + &  & + & + & + &  & + & + & + & \\
+ & + & + &  &  &  &  & + &  &  &  & + &  &  &  & +\\
 & + &  &  & + &  & + & + & + &  & + & + & + &  & + & +\\
+ &  & + & + &  & + &  &  &  & + &  &  &  & + &  & \\
+ & + &  & + &  &  & + &  &  &  & + &  &  &  & + & \\
 &  & + &  & + & + &  & + & + & + &  & + & + & + &  & +\\
+ &  &  &  &  & + & + & + &  & + & + & + &  & + & + & +\\
 & + & + & + & + &  &  &  & + &  &  &  & + &  &  & \\
\end{array}\right]}.\]
To form our regular simplex we now choose all vectors with 0 in the last coordinate, discard the last element, replace every 0 with a -1, and then scale by $\frac{1}{\sqrt{15}}$ giving us the vectors (given by rows):
\[\begin{aligned}
X_2 = \frac{1}{\sqrt{15}}\setlength{\arraycolsep}{.8pt}\scalefont{.4}{\left[\begin{array}{rrrrrrrrrrrrrrrr}
- & - & - & + & - & - & - & + & - & - & - & + & + & + & + \\
+ & - & + & + & + & - & + & + & + & - & + & + & - & + & - \\
+ & + & - & + & + & + & - & + & + & + & - & + & - & - & + \\
- & + & + & + & - & + & + & + & - & + & + & + & + & - & - \\
+ & + & + & - & - & - & - & + & + & + & + & - & + & + & + \\
- & + & - & - & + & - & + & + & - & + & - & - & - & + & - \\
- & - & + & - & + & + & - & + & - & - & + & - & - & - & + \\
+ & - & - & - & - & + & + & + & + & - & - & - & + & - & - \\
+ & + & + & - & + & + & + & - & - & - & - & + & + & + & + \\
- & + & - & - & - & + & - & - & + & - & + & + & - & + & - \\
- & - & + & - & - & - & + & - & + & + & - & + & - & - & + \\
+ & - & - & - & + & - & - & - & - & + & + & + & + & - & - \\
- & - & - & + & + & + & + & - & + & + & + & - & + & + & + \\
+ & - & + & + & - & + & - & - & - & + & - & - & - & + & - \\
+ & + & - & + & - & - & + & - & - & - & + & - & - & - & + \\
- & + & + & + & + & - & - & - & + & - & - & - & + & - & - \\
\end{array}\right]}.
\end{aligned}\]
Likewise each $Q_j(v)$ gives us a coset of size 32, which in turn is transformed to a regular simplex $X_j$ in this manner.
\subsubsection*{Symmetric design parameters}
The parameters of this scheme will be $v = 2^{2r}$, $k = 2^{r-1}\left(2^{r}+1\right)$, $\lambda = 2^{r-1}\left(2^{r-1}+1\right)$, $s = 2^{r-1}$, $\nu = 2^{r-2}\left(2^r+1\right)$,  and $\mu=2^{r-2}\left(2^r+3\right)$.
\subsubsection*{Intersection numbers}
We list the unique intersection numbers omitting the trivial $p_{0,i}^j$ parameters. Note that each $p^{j}_{i,k}$ is scaled by a constant based on $i$ and $k$ given in the first row of our table.
\[\begin{array}{c|ccc|cc|c}
j & \nicefrac{p^{j}_{1,1}}{2^{r-2}}			& \nicefrac{p^{j}_{1,2}}{2^{r-1}} 			& \nicefrac{p^{j}_{1,3}}{2^{r-2}} & p^{j}_{2,2} & \nicefrac{p^{j}_{2,3}}{2^{r-1}} & \nicefrac{p^{j}_{3,3}}{2^{r-2}}\\\hline
0 & \left(2^{r+1}+2\right)(w-1)	& 0 							& 0 					& 2^{2r}-1	& 0   							& (2^{r+1}-2)(w-1)\\
1 & \left(2^r+3\right)(w-2)		& \left(2^{r}+1\right)-2^{1-r}	& (2^r-1)(w-2) 			& 0   		& \left(2^r-1\right)   			& \left(2^{r-2}-1\right)(w-2)\\
2 & \left(2^{r}+2\right)(w-1)	& 0 							& 2^{r}(w-1)			& 2^{2r}-2 	& 0 							& (2^{r}-2)(w-1)\\
3 & \left(2^r+1\right)(w-2) 	& 2^{r}+1						& (2^r+1)(w-2) 			& 0   		&\left(2^{r}-1\right)-1 		& \left(2^{r-2}-3\right)(w-2)\\
\end{array}\]
\subsubsection*{Krein parameters}
As with the intersection numbers, we list each unique Krein parameter omitting the trivial $q_{0,i}^j$ parameters. No scaling is needed here.
\[\begin{array}{c|ccc|cc|c}
j & q^{j}_{1,1}			& q^{j}_{1,2}	 						& q^{j}_{1,3} 	& q^{j}_{2,2} 									& q^{j}_{2,3} 	& q^{j}_{3,3}\\\hline
0 & 2^{2r}-1			& 0										& 0 			& (w-1)(2^{2r}-1)								& 0  			& w-1\\
1 & \frac{2^{2r}}{w}-2	& 2^{2r}\left(\frac{w-1}{w}\right)		& 0 			& 2^{2r}\left(\frac{(w-1)^2}{w}\right)-2(w-1)	& w-1			& 0\\
2 & \frac{2^{2r}}{w}	& 2^{2r}\left(\frac{w-1}{w}\right)-2	& 1				& 2^{2r}\left(\frac{(w-1)^2}{w}\right)+2(w-2)	& w-2			& 0\\
3 & 0 					& 2^{2r}-1								& 0 			& (w-2)(2^{2r}-1)								& 0  			& w-2\\
\end{array}\]

\section{Menon parameters}
\label{LSSDMUBs}
Recall from Section $\ref{bases}$ that we found the Gram matrix for a set of bases by adding the first two idempotents of our scheme, giving us
\begin{equation}\label{MubGram}M = w(E_0+E_1) = \left(A_0+\frac{v-k+\sqrt{k-\lambda}}{v\sqrt{k-\lambda}}A_1-\frac{k-\sqrt{k-\lambda}}{v\sqrt{k-\lambda}}A_3\right)\end{equation}
where $k$ is the block size of the $\mu$-heavy LSSD. If $\frac{v-k+\sqrt{k-\lambda}}{v\sqrt{k-\lambda}}$ and $-\frac{k-\sqrt{k-\lambda}}{v\sqrt{k-\lambda}}$ have the same absolute value, then $M$ is the Gram matrix for a set of $w$ mutually unbiased bases in $\mathbb{R}^v$. This will only occur when
\[\begin{aligned}v-2k&=-2\sqrt{k-\lambda}.\\
\end{aligned}\]
This means our LSSD must be optimistic, leading to the following lemma:
\begin{lemma}
	\label{twoside}
	Let $\mathbb{A}$ be Bose-Mesner algebra of an optimistic $LSSD(v,k,\lambda;w)$ with $Q$-polynomial ordering $E_0,E_1,E_2,E_3$ of its primitive idempotents. If $\vert v-2k\vert = 2\sqrt{k-\lambda}$ then $w(E_0+E_1)$ is the Gram matrix of a set of $w$ MUBs in dimension $v$.\qed
\end{lemma}
It is important to note here that there exist pessimistic LSSDs such that $v-2k = 2\sqrt{k-\lambda}$. One such example is a specific case of the degenerate parameters in \ref{degenerate} given by $(v,k,\lambda)=(4,1,0)$ with the graph $\Gamma_1$ displayed below.
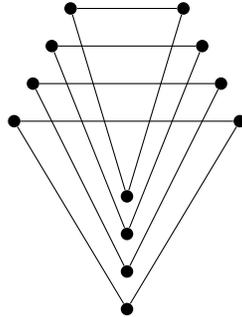
\[\begin{tikzpicture}[scale = .5,node distance=2cm,
thin,main node/.style={circle,fill=black,scale = .5}]

\node[main node] (11) at (-3,1) {};
\node[main node] (12) at (-2.5,2) {};
\node[main node] (13) at (-2,3) {};
\node[main node] (14) at (-1.5,4) {};
\node[main node] (24) at (1.5,4) {};
\node[main node] (23) at (2,3) {};
\node[main node] (22) at (2.5,2) {};
\node[main node] (21) at (3,1) {};
\node[main node] (34) at (0,-1) {};
\node[main node] (33) at (0,-2) {};
\node[main node] (32) at (0,-3) {};
\node[main node] (31) at (0,-4) {};

\draw [-] (11) -- (21) -- (31) -- (11);
\draw [-] (12) -- (22) -- (32) -- (12);
\draw [-] (13) -- (23) -- (33) -- (13);
\draw [-] (14) -- (24) -- (34) -- (14);
\end{tikzpicture}\]
We can easily see that $v-2k = 2 = 2\sqrt{k-\lambda}$. However, the sum of the first two eigenspaces gives
\[M = 3(E_0 + E_1) =\frac{1}{3}A_0 -\frac{1}{4}A_1\]
which is not the Gram matrix of a set of MUBs. In fact, any Menon parameter set with $v/4$ odd will satisfy $\vert v-2k\vert = 2\sqrt{k-\lambda}$ however none of these will produce MUBs. Therefore our restriction to optimistic LSSDs is required and we cannot say that any LSSD satisfying $\vert v-2k\vert = 2\sqrt{k-\lambda}$ will give us MUBs using this construction.
\subsection{Projective Restriction}
We now take a closer look at our restriction $v-2k = -2\sqrt{k-\lambda}$. First note we can square both sides to get
\[\begin{aligned}
4(k-\lambda)&=v^2-4k(v-k).
\end{aligned}\]
Using \eqref{sym:2}, this gives $v=4(k-\lambda)$ where we apply \eqref{sym:1} to get
\[k^2+k+\lambda=4\lambda(k-\lambda).\]
Solving this for $k$ gives $k=\frac{4\lambda+1}{2}\pm\frac{\sqrt{4\lambda+1}}{2}$ requiring $\frac{1}{2}\pm\frac{1}{2}\sqrt{4\lambda+1}$ to be an integer. Therefore $\sqrt{4\lambda+1}$ must be an odd integer. Assume $4\lambda+1 = (2u-1)^2$ for some positive integer $u$. Then $\lambda = u^2-u$ and
\[\begin{aligned}
k&=2u^2-(2\mp 1)u+\left(\frac{1\mp1}{2}\right).\\
\end{aligned}\]
If we re-parameterize the second family to avoid the trivial $(0,0,0)$ design when $u=1$, we get the complementary families:
\[\begin{aligned}
\lambda &= (u-1)u\qquad\qquad &\lambda^\prime &= (u+1)u\\
k&=(2u-1)u\qquad \text{and}\qquad&k^\prime &= (2u+1)u\\
v&= 4u^2\qquad \qquad&v^\prime &= 4u^2.
\end{aligned}\]
As we are restricted to optimistic LSSDs, we must use the second family for our $\mu$-heavy LSSD and the first will arise in the $\nu$-heavy complement.
\subsection{Combinatorial Restrictions}
We now return to the original graph $\Gamma$ and consider feasibility conditions based on $\mu$ and $\nu$. From \eqref{mu-nu},
\[\begin{aligned}
s=\sqrt{k-\lambda}=u,\qquad\nu =\frac{k(k-s)}{v}=u^2+\frac{u}{2},\qquad\mu=\nu+s=u^2+\frac{3}{2}u.
\end{aligned}\]
Therefore $\nu$ (and $\mu$) are integral if and only if $u$ is even, resulting in the following theorem
\begin{theorem}
	\label{MubLssdrestrictions}
	Let $\Gamma$ be an optimistic $LSSD(v,k,\lambda;w)$. If $v-2k = -2\sqrt{k-\lambda}$ then:
	\begin{enumerate}[(i)]
		\item $v = 4u^2$
		\item $k = 2u^2+u$
		\item $\lambda = u^2+u$
		\item $w\leq 2u^2$
	\end{enumerate}
	Further, if $u$ is odd, then $w = 2$.\qed
\end{theorem}
One consequence of this result is that we cannot take an arbitrary set of MUBs and project the space down one dimension to get the adjacency matrix of a linked system. This can be seen when the dimension is 4 times an odd square. For example in $\mathbb{R}^4$ we can achieve 3 MUBs. However, any such projection would gives us an optimistic linked system with $v=4$ and $w = 3$. Our final statement in our theorem tells us that since $v = 4u^2$ with odd $u$, $w$ must equal 2. The upper bound for $w$ is achieved whenever $u$ is a power of two using the Cameron Seidel scheme. The question of a better upper bound for $u$ even but not a power of two is still open. It was conjectured that the upper bound of $w$ will depend solely on the highest power of $2$ which divides $v$ since the same conjecture was made for MUBs. However, we will show later that this is not true as we can find a sequence of parameter sets (namely Menon with $v = 16(2t+1)^4$) for which $w\rightarrow \infty$ as $t\rightarrow \infty$.
\subsection{Hadamard Matrix Equivalence}
We have seen that, while we can build MUBs from certain LSSDs, we cannot always build LSSDs from MUBs. In this section, we establish an equivalence between these LSSDs with sets of regular unbiased Hadamard matrix (see \cite{Kharaghani} for more detailed information on unbiased Hadamard matrices). A real Hadamard matrix is a $v\times v$ matrix $H$ with entries $\pm 1$ such that $HH^T = vI$. $H$ is a \textit{regular Hadamard matrix} if $HJ = cJ$ for some constant $c$. Two Hadamard matrices $H_1$ and $H_2$ are \textit{unbiased} if $\frac{1}{v}H_1H_2^T$ is itself a Hadamard matrix. Finally, a set of Hadamard matrix matrices are unbiased if each pair is unbiased. Using these definitions, consider the following:
\begin{theorem}
	\label{LSSDtoHad}
	Let $\Gamma$ be an optimistic $LSSD(v,k,\lambda;w)$. If $\vert v-2k\vert = -2\sqrt{k-\lambda}$, then there exists a set of $w-1$ real unbiased regular Hadamard matrices.
\end{theorem}
\begin{proof}
	Let $(X,\mathcal{R})$ be the association scheme arising from $\Gamma$ with Bose-Mesner algebra $\mathbb{A}$. Let $\left\{E_0,E_1,E_2,E_3\right\}$ be the set of idempotents under the natural $Q$-polynomial ordering. From Lemma \ref{twoside}, $G = w(E_0+E_1)$ is the Gram matrix of a set of $w$ MUBs in $\mathbb{R}^v$. Let the $w$ MUBs be given by the columns of the $w$ unitary matrices $\left\{M_1,\dots,M_w\right\}$ and without loss of generality assume $M_1 = I$. Then any column from another $M_i$ $(i\neq 1)$ must have entries $\pm \frac{1}{\sqrt{v}}$. For $1<i\leq w$, let $H_i = \sqrt{v}M_i$. First note that $H_iH_i^T = vM_iM_i^T = vI$, therefore for each $1<i\leq w$, $H_i$ is a Hadamard matrix. Now consider that the first $v$ rows of $G$ will have the block form
	$\left[\begin{array}{ccccc}
	I & M_2 & M_3 & \dots & M_{w}\\
	\end{array}\right].$
	However from \eqref{MubGram}, we have that the positive (resp., negative) entries of $M_i$ represent adjacency between vertices in the first and $i^\text{th}$ fibers of the $\mu$-heavy (resp., $\nu$-heavy) LSSD. Therefore each $M_i$ must have $k$ positive entries and $v-k$ negative entries giving that $H_i$ must be regular. Now define $\binom{w}{2}$ matrices $M_{i,j}$ where $M_{i,j}$ is the unitary matrix representing Basis $j$ when Basis $i$ is taken to be the standard basis (so $M_{1,j} = M_j$). Then we repeat all previous arguments to show that $G$ has the block form:
	\[G = \left[\begin{array}{ccccccc}
	I & M_{1,2} & \dots & M_{1,w}\\
	M_{2,1} & I & \dots & M_{2,w}\\
	\vdots & \vdots  & \ddots & \vdots\\
	M_{w,1} & M_{w,2} & \dots & I\\
	\end{array}\right]\]
	where $\sqrt{n}M_{i,j}$ is a Hadamard matrix for all $1\leq i\neq j\leq w$. Now consider a second association scheme $\mathcal{L}^\prime$ arising from the subgraph of $\Gamma$ induced on three distinct fibers $X_i$, $X_j$, and $X_k$. The matrix $G^\prime = w(E_0^\prime + E_1^\prime)$ will have the form:
		\[G^\prime = \left[\begin{array}{ccccccc}
	I & M_{i,j} & M_{i,k}\\
	M_{j,i} & I  & M_{j,k}\\
	M_{k,i} & M_{k,j} & I\\
	\end{array}\right].\]
	Noting that $G^2 = wG$, the block in the $(1,2)$ block of $G^2$ gives us that
	\[2M_{i,j} + M_{i,k}M_{k,j} = 2M_{i,j}\qquad \text{ or } \qquad M_{i,k}M_{k,j} = M_{i,j}.\]
	Therefore, if we return to the original $LSSD$ and define $H_{i,j} = \sqrt{v}M_{i,j}$ then we find that $\frac{1}{\sqrt{v}}H_i^TH_j = H_{i,j}$. Therefore the set $\left\{H_2,\dots,H_w\right\}$ is a set of $w-1$ regular unbiased Hadamard matrices.
\end{proof}
We now seek to show the converse:
\begin{theorem}
	Assume $w>2$. Let $\left\{H_2,\dots,H_w\right\}$ be $w-1$ regular unbiased Hadamard matrices of side length $v$. Then there exists an optimistic $LSSD(v,k,\lambda;w)$.
\end{theorem}
\begin{proof}
	Assume without loss of generality that the row sum of each of our Hadamard matrices is positive. Define vectors $x_{i,j}$ for $2\leq i\leq w$ and $1\leq j\leq v$ such that $x_{i,j}$ is the $j^\text{th}$ column of $H_i-\frac{1}{\sqrt{v}}J$. Let $x_{1,j}$ be the $j^\text{th}$ column of $\sqrt{v}I-\frac{1}{\sqrt{v}}{J}$. Note that for all $1\leq i\leq w$, $\vnorm{x_{i,j}} = v-1$. Then, for all $i,j$, let $\hat{x}_{i,j} = \frac{x_{i,j}}{\sqrt{v-1}}$. Letting $X_i = \left\{x_{i,1},\dots,x_{i,v}\right\}$, we claim that $\left\{X_1,\dots, X_w\right\}$ is a set of linked simplices. To show this, fix $j\neq k$, $i\neq i'$ and consider the following four inner products:
	\begin{align}
	\label{had:1}\left<\hat{x}_{1,j},\hat{x}_{1,k}\right> &= \frac{1}{v-1}\left(vI-J\right)_{j,k} = -\frac{1}{v-1},\\
	\label{had:2}\left<\hat{x}_{i,j},\hat{x}_{i,k}\right> &=\frac{1}{v-1}(H_iH_i^T -J)_{j,k} = -\frac{1}{v-1},\\
	\label{had:3}\left<\hat{x}_{1,j},\hat{x}_{i,k}\right> &=\frac{\sqrt{v}}{v-1}\left(H_i^T-\frac{1}{\sqrt{v}}J\right)_{j,k},\\
	\label{had:4}\left<\hat{x}_{i,j},\hat{x}_{i',k}\right> &=\frac{\sqrt{v}}{v-1}\left(\frac{1}{\sqrt{v}}H_iH_{i^\prime}^T -\frac{1}{\sqrt{v}}J\right)_{j,k}.
	\end{align}
	\eqref{had:1} and \eqref{had:2} give us the inner products within each $X_i$. Since $\frac{1}{\sqrt{v}}H_iH_{i^\prime}^T$ is a Hadamard matrix, \eqref{had:3} and \eqref{had:4} tell us that inner products between sets $X_i$ and $X_{i^\prime}$ take values of $\frac{\pm\sqrt{v}-1}{v-1}$. Finally note that the all ones vector is orthogonal to all $x_{i,j}$. Therefore $\left\{X_1,\dots,X_w\right\}$ is a set of $w$ simplices in $\mathbb{R}^{v-1}$ such that inner products between simplices can take only two possible values. Finally consider that the possible inner products are $\frac{\sqrt{v}}{v-1}\left(\pm 1 - \frac{1}{\sqrt{v}}\right)$. This tells us that $\vert \gamma\vert <\vert \zeta\vert$ where $\gamma$ is the positive inner product and $\zeta$ is the negative. Therefore, since the centroid of any simplex is the origin, we must have more positive inner products between simplices than negative, telling us our $LSSD$ is optimistic.
\end{proof}
This leaves us with the following theorem
\begin{theorem}
	\label{equiv}
	An optimistic $LSSD(v,k,\lambda;w)$ with $\vert v-2k\vert = 2\sqrt{k-\lambda}$ exists if and only if there exists $w-1$ regular unbiased Hadamard matrices with side length $v$.\qed
\end{theorem}
\subsection{Constructing LSSDs from certain MUBs}
Using the results from the last section and the close relation between MUBs and Hadamard matrices, we wish to build new LSSDs. From theorem \ref{equiv} and theorem \ref{MubLssdrestrictions}, we are only going to find optimistic LSSDs with Menon parameters. Goethals gives a construction in \cite{Cameron} for $w = 2u^2$ whenever $u$ is a power of 2 (see \ref{kerdock}). Therefore we skip this case and instead look for constructions where $u$ (and equivalently $v$) is not necessarily a power of 2.
\subsubsection{Beth and Wocjan construction}
Beth and Wocjan in \cite{Wocjan} detail a way to create MUBs from MOLS. They take a set of $t$ MOLS with side length $d$ and create $t+2$ MUBs in dimension $d^2$. The process is to convert the MOLS into an orthogonal array with $d^2$ rows. They then expand the array by replacing each column with $d$ columns giving by the characteristic vector of each symbol in that column. Finally, they extend this Matrix by replacing each 1 in the array with a row from a Hadamard matrix matrix and each 0 by an appropriate length vector of 0s. The result is that the $d$ columns arising from each original column are orthogonal to each other. We will focus on the case where the resultant MUBs produce regular Hadamards using their inner products.\newpar
For our purposes, an orthogonal array of size $(n^2\times s)$ has entries from the set $\left\{1,\dots, n\right\}$ and any two columns contain each ordered pair exactly once. Let $O$ be an orthogonal array of size $n^2\times s$, let $C^i$ denote the $i^{th}$ column of $O$ with entries $C^i_k$ ($1\leq k\leq n^2$). We may uniquely express
\[C^i = \sum_{j=1}^n kB^{i,j} \]
where each $B^{i,j}$ is a 01-vector of length $n^2$. As each symbol $j$ appears in each column $C^i$ exactly $n$ times, $B^{i,j}$ will have $n$ 1s and $n^2-n$ 0s. Let $H$ be a Hadamard matrix matrix of size $n\times n$. For $1\leq l\leq n$ define a matrix $M^{i,j,l}$ as follows: For $1\leq k\leq n$, we replace the $k^\text{th}$ 1, counting from the top, in $B^{i,j}$ with the $H_{k,l}$. This produces $n^2s$ columns each with $n^2$ entries $M_k^{i,j,l}\in\left\{0,1,-1\right\}$.\newpar
The fact that $H^TH = nI$ together with $B^{i,j}\circ B^{i,j^\prime} = 0$ for $j\neq j^\prime$ give us that $\mathcal{B}_i = \left\{M^{i,1,1},\dots,M^{i,n,n}\right\}$ is an orthogonal basis for each $i=1,\dots,s$. Each vector in these bases has squared norm $n$. For $i\neq i'$, $C^i$ and $C^{i'}$ denote distinct columns in our orthogonal array and therefore for any $j$ and $j'$ (not necessarily distinct), $B^{i,j}\circ B^{i',j'}$ has one nonzero entry. Therefore $M^{i,j,l}\circ M^{i',j',l'}$ also has one nonzero entry. Therefore
\begin{equation}\left<M^{i,j,l},M^{i',j',l'}\right>=M^{i,j,l}_kM^{i',j',l'}_k=\pm 1\\
\end{equation}
where $O_{k,i} = j$ and $O_{k,i'} = j'$. Therefore the bases $\mathcal{B}_1,\dots,\mathcal{B}_s$ produced by Beth and Wocjan are unbiased.\newpar
We now show that if $H$ is regular, then the resulting unbiased Hadamard matrices are regular (see proof of Theorem \ref{LSSDtoHad} for the construction of these Hadamard matrices). For each Hadamard matrix, the row sum is the sum of inner products between a column $M^{i,j,l}$ of one basis with the set of $n^2$ columns $M^{i',j',l'}$ ($i\neq i'$, $1\leq j',l'\leq n)$ of the second basis used in its construction. We first sum $\left<M^{i,j,l},M^{i',j',l'}\right>$ over $l'$ to get
\begin{equation*}
\sum_{l'} \left<M^{i,j,l},M^{i',j',l'}\right>=M^{i,j,l}_k\left(\sum_{l'}M^{i',j',l'}_k\right)=pM^{i,j,l}_k.
\end{equation*}
We then sum this result over $j'$ noting that $k$ was chosen so that $O_{k,i} = j$ and $O_{k,i'} = j'$, and therefore depends on $j'$. As this sum will include every nonzero entry in $M^{i,j,l}$ exactly once, we know
\begin{equation}
\label{rowsum}
\sum_{j'}\sum_{l'}\left<M^{i,j,l},M^{i',j',l'}\right> =\sum_{j'}pM^{i,j,l}_{k}=p\sum_{j'}H_{j',l}=p^2.
\end{equation}
Then the sum of any row of the Hadamard matrix built from $M^{i,j,l}$ and $M^{i',j',l'}$ ($i\neq i'$) will be $p^2$. Further, we showed in Theorem \ref{LSSDtoHad} that these Hadamard matrices are unbiased. Noting that $n = 4t^2$ for some $t$, as our original Hadamard must be regular, Theorem \ref{equiv} tells us that our LSSD will be an optimistic $LSSD(16t^4,k,\lambda;s)$. This leads to our final theorem:
\begin{theorem}
	\label{newLSSD}
	Given a regular Hadamard matrix of order $s$ and an orthogonal array of size $s^2\times N$,
	\begin{itemize}
		\item There exists $N-1$ regular unbiased Hadamard matrices of order $s^2$.
		\item There exists a $LSSD$ with $v=s^2$ and $w=N$.
	\end{itemize}\qed
\end{theorem}
\begin{corollary}
	\label{asymptotics}
	For sufficiently large $s$, if there exists a regular Hadamard matrix of order $s$, then there exists a $LSSD(s^2,k,\lambda; w)$ with $w \geq s^\frac{1}{14.8}$.
\end{corollary}
\begin{proof}
	\cite{Colbourn} states that for sufficiently large $s$, $N(s)\geq s^\frac{1}{14.8}$ where $N(s)$ is the maximum number of columns in an orthogonal array $s$ on $s$ symbols.
\end{proof}
\begin{corollary}
	For any $n\geq 1$ and $w>2$, there exists an odd $t$ permitting a $LSSD(16^nt,k,\lambda;w)$.
\end{corollary}
\begin{proof}
	\cite{Xiang} tells us that for any odd $t$, there exists a regular Hadamard matrix of order $4t^4$. Let $H_t$ be the regular Hadamard matrix of order $4t^4$. Using Corollary \ref{asymptotics}, we can choose $t$ large enough to guarantee the existence of a $LSSD(16t^8,k,\lambda;w)$. Now consider the Hadamard matrix
	\[H = \left[\begin{array}{rrrr}
	-1 & 1 & 1 & 1\\
	1 & -1 & 1 & 1\\
	1 & 1 & -1 & 1\\
	1 & 1 & 1 & -1\\
	\end{array}\right].\]
	Using this matrix, we can now build the regular Hadamard matrix $H_{n,t} = H_t\otimes^{n-1} H$ which is regular of order $4^nt^4$. This matrix, again paired with Corollary \ref{asymptotics}, now guarantees the existence of a $LSSD(16^nt^8,k,\lambda;w)$ for any choice of $n$.
	\end{proof}
	\begin{corollary}
	There exists an $LSSD(v,k,\lambda;w)$ with $v=36^{2n}$ and $w = 4^n+1$ for all $n\geq 1$.
\end{corollary}
\begin{proof}
	Using the MacNeish construction (\cite{Macneish},\cite[Thm~1.1.2]{Bommel}), there exists an orthogonal array $O_n$ of size $36^{2n}\times(4^n+1)$. Consider the regular Hadamard matrix of order 36:
	\[H = \setlength{\arraycolsep}{0.6pt}\scalefont{.4}{\left[\begin{array}{cccccccccccccccccccccccccccccccccccc}
		-&-&-&-&+&-&-&+&+&+&+&-&+&+&-&+&-&+&+&+&+&+&-&-&-&+&+&+&+&-&+&+&+&-&+&-\\
		+&-&-&-&-&+&-&-&+&+&-&+&+&-&+&-&+&+&+&+&+&-&-&-&+&+&+&+&-&+&+&+&-&+&-&+\\
		+&+&-&-&-&-&+&-&-&-&+&+&-&+&-&+&+&+&+&+&-&-&-&+&+&+&+&-&+&+&+&-&+&-&+&+\\
		-&+&+&-&-&-&-&+&-&+&+&-&+&-&+&+&+&-&+&-&-&-&+&+&+&+&+&+&+&+&-&+&-&+&+&-\\
		-&-&+&+&-&-&-&-&+&+&-&+&-&+&+&+&-&+&-&-&-&+&+&+&+&+&+&+&+&-&+&-&+&+&-&+\\
		+&-&-&+&+&-&-&-&-&-&+&-&+&+&+&-&+&+&-&-&+&+&+&+&+&+&-&+&-&+&-&+&+&-&+&+\\
		-&+&-&-&+&+&-&-&-&+&-&+&+&+&-&+&+&-&-&+&+&+&+&+&+&-&-&-&+&-&+&+&-&+&+&+\\
		-&-&+&-&-&+&+&-&-&-&+&+&+&-&+&+&-&+&+&+&+&+&+&+&-&-&-&+&-&+&+&-&+&+&+&-\\
		-&-&-&+&-&-&+&+&-&+&+&+&-&+&+&-&+&-&+&+&+&+&+&-&-&-&+&-&+&+&-&+&+&+&-&+\\
		+&+&-&+&+&-&+&-&+&+&+&+&+&-&+&+&-&-&+&-&+&-&+&+&+&-&+&-&-&-&+&+&+&-&-&-\\
		+&-&+&+&-&+&-&+&+&-&+&+&+&+&-&+&+&-&-&+&-&+&+&+&-&+&+&-&-&+&+&+&-&-&-&-\\
		-&+&+&-&+&-&+&+&+&-&-&+&+&+&+&-&+&+&+&-&+&+&+&-&+&+&-&-&+&+&+&-&-&-&-&-\\
		+&+&-&+&-&+&+&+&-&+&-&-&+&+&+&+&-&+&-&+&+&+&-&+&+&-&+&+&+&+&-&-&-&-&-&-\\
		+&-&+&-&+&+&+&-&+&+&+&-&-&+&+&+&+&-&+&+&+&-&+&+&-&+&-&+&+&-&-&-&-&-&-&+\\
		-&+&-&+&+&+&-&+&+&-&+&+&-&-&+&+&+&+&+&+&-&+&+&-&+&-&+&+&-&-&-&-&-&-&+&+\\
		+&-&+&+&+&-&+&+&-&+&-&+&+&-&-&+&+&+&+&-&+&+&-&+&-&+&+&-&-&-&-&-&-&+&+&+\\
		-&+&+&+&-&+&+&-&+&+&+&-&+&+&-&-&+&+&-&+&+&-&+&-&+&+&+&-&-&-&-&-&+&+&+&-\\
		+&+&+&-&+&+&-&+&-&+&+&+&-&+&+&-&-&+&+&+&-&+&-&+&+&+&-&-&-&-&-&+&+&+&-&-\\
		+&+&+&+&-&-&-&+&+&-&+&-&+&-&-&-&+&-&+&+&+&+&-&+&+&-&-&+&+&-&+&-&+&+&-&+\\
		+&+&+&-&-&-&+&+&+&+&-&+&-&-&-&+&-&-&-&+&+&+&+&-&+&+&-&+&-&+&-&+&+&-&+&+\\
		+&+&-&-&-&+&+&+&+&-&+&-&-&-&+&-&-&+&-&-&+&+&+&+&-&+&+&-&+&-&+&+&-&+&+&+\\
		+&-&-&-&+&+&+&+&+&+&-&-&-&+&-&-&+&-&+&-&-&+&+&+&+&-&+&+&-&+&+&-&+&+&+&-\\
		-&-&-&+&+&+&+&+&+&-&-&-&+&-&-&+&-&+&+&+&-&-&+&+&+&+&-&-&+&+&-&+&+&+&-&+\\
		-&-&+&+&+&+&+&+&-&-&-&+&-&-&+&-&+&-&-&+&+&-&-&+&+&+&+&+&+&-&+&+&+&-&+&-\\
		-&+&+&+&+&+&+&-&-&-&+&-&-&+&-&+&-&-&+&-&+&+&-&-&+&+&+&+&-&+&+&+&-&+&-&+\\
		+&+&+&+&+&+&-&-&-&+&-&-&+&-&+&-&-&-&+&+&-&+&+&-&-&+&+&-&+&+&+&-&+&-&+&+\\
		+&+&+&+&+&-&-&-&+&-&-&+&-&+&-&-&-&+&+&+&+&-&+&+&-&-&+&+&+&+&-&+&-&+&+&-\\
		+&+&-&+&+&+&-&+&-&+&+&+&-&-&-&+&+&+&-&-&+&-&+&-&-&+&-&+&+&+&+&-&+&+&-&-\\
		+&-&+&+&+&-&+&-&+&+&+&-&-&-&+&+&+&+&-&+&-&+&-&-&+&-&-&-&+&+&+&+&-&+&+&-\\
		-&+&+&+&-&+&-&+&+&+&-&-&-&+&+&+&+&+&+&-&+&-&-&+&-&-&-&-&-&+&+&+&+&-&+&+\\
		+&+&+&-&+&-&+&+&-&-&-&-&+&+&+&+&+&+&-&+&-&-&+&-&-&-&+&+&-&-&+&+&+&+&-&+\\
		+&+&-&+&-&+&+&-&+&-&-&+&+&+&+&+&+&-&+&-&-&+&-&-&-&+&-&+&+&-&-&+&+&+&+&-\\
		+&-&+&-&+&+&-&+&+&-&+&+&+&+&+&+&-&-&-&-&+&-&-&-&+&-&+&-&+&+&-&-&+&+&+&+\\
		-&+&-&+&+&-&+&+&+&+&+&+&+&+&+&-&-&-&-&+&-&-&-&+&-&+&-&+&-&+&+&-&-&+&+&+\\
		+&-&+&+&-&+&+&+&-&+&+&+&+&+&-&-&-&+&+&-&-&-&+&-&+&-&-&+&+&-&+&+&-&-&+&+\\
		-&+&+&-&+&+&+&-&+&+&+&+&+&-&-&-&+&+&-&-&-&+&-&+&-&-&+&+&+&+&-&+&+&-&-&+
		\end{array}\right]}.\]
	Since $H$ is regular, $H_n = H^{\otimes n}$ is a regular Hadamard of order $36^n$. Then $O_n$ and $H_n$, along with Theorem \ref{newLSSD}, give us the desired LSSD.
\end{proof}
The same construction gives, for example, $LSSD(100^{2n},k,\lambda;4^n+1)$ for all $n\geq 1$. Finally, we note that if we can build a regular Hadamard matrix of order $4t^2$ for $1\leq t\leq 50$, the table of largest known orthogonal arrays for small $n$ in \cite{Colbourn} gives us LSSDs with the following number of fibers.\newpar
\begin{tabular}{c|*{17}{c}}
	$t$ & 1 & 2 & 3 & 4 & 5 & 6 & 7 & 8 & 9 & 10& 11 & 12 & 13 & 14 & 15 & 16 & 17 \\\hline
	$w$&5&17&9&65&10&12&8&257&10&17&17&10&10&10&29&1025&10\\\\
	$t$ & 18 & 19 & 20& 21 & 22 & 23 & 24 & 25 & 26 & 27 & 28 & 29 & 30& 31 & 32 & 33 \\\hline
	$w$&26&11&26&11&17&11&32&10&17&10&50&30&30&12&4097&32\\\\
	$t$ & 34 & 35 & 36 & 37 & 38 & 39 & 40 & 41 & 42 & 43 & 44 & 45 & 46 & 47 & 48 & 49 &\\\hline
	$w$&18&32&65&32&18&32&26&13&20&32&65&17&32&32&30&17
\end{tabular}\newpar
To give an example of the construction for Theorem $\ref{newLSSD}$ we build a $\text{LSSD}(16,10,6; 3)$. In all matrices that follow, ``$+$" denotes a positive 1, ``$-$" denotes a $-1$, and an empty space denotes a $0$. We begin using the orthogonal array given by $O$ and the Hadamard matrix $H$ shown below:
\[O^T = \left[\begin{array}{cccccccccccccccc}
1 & 1 & 1 & 1 & 2 & 2 & 2 & 2 & 3 & 3 & 3 & 3 & 4 & 4 & 4 & 4\\
1 & 2 & 3 & 4 & 1 & 2 & 3 & 4 & 1 & 2 & 3 & 4 & 1 & 2 & 3 & 4\\
1 & 2 & 3 & 4 & 2 & 3 & 4 & 1 & 3 & 4 & 1 & 2 & 4 & 1 & 2 & 3\\
\end{array}\right],\qquad H = \left[\begin{array}{rrrr}
- & + & + & +\\
+ & - & + & +\\
+ & + & - & +\\
+ & + & + & -\\
\end{array}\right].\]
Using this OA, we have
\[B^{:,:} = \setlength{\arraycolsep}{0.8pt}\scalefont{.4}{\left[\begin{array}{cccc|cccc|cccc}
+&&& &+&&& &+&&&\\
+&&& &&+&& &&+&&\\
+&&& &&&+& &&&+&\\
+&&& &&&&+ &&&&+\\
&+&& &+&&& &&+&&\\
&+&& &&+&& &&&+&\\
&+&& &&&+& &&&&+\\
&+&& &&&&+ &+&&&\\
&&+& &+&&& &&&+&\\
&&+& &&+&& &&&&+\\
&&+& &&&+& &+&&&\\
&&+& &&&&+ &&+&&\\
&&&+ &+&&& &&&&+\\
&&&+ &&+&& &+&&&\\
&&&+ &&&+& &&+&&\\
&&&+ &&&&+ &&&+&\\
\end{array}\right]},\]
and then using our Hadamard matrix,
\[\begin{aligned}
M^{1,:,:}=\setlength{\arraycolsep}{0.6pt}\scalefont{.4}{\left[\begin{array}{rrrr|rrrr|rrrr|rrrr}
- & + & + & +&&&&&&&&&&&&\\
+ & - & + & +&&&&&&&&&&&&\\
+ & + & - & +&&&&&&&&&&&&\\
+ & + & + & -&&&&&&&&&&&&\\
&&&&- & + & + & +&&&&&&&&\\
&&&&+ & - & + & +&&&&&&&&\\
&&&&+ & + & - & +&&&&&&&&\\
&&&&+ & + & + & -&&&&&&&&\\
&&&&&&&&- & + & + & +&&&&\\
&&&&&&&&+ & - & + & +&&&&\\
&&&&&&&&+ & + & - & +&&&&\\
&&&&&&&&+ & + & + & -&&&&\\
&&&&&&&&&&&&- & + & + & +\\
&&&&&&&&&&&&+ & - & + & +\\
&&&&&&&&&&&&+ & + & - & +\\
&&&&&&&&&&&&+ & + & + & -\\
\end{array}\right]},\quad
M^{2,:,:}&=\setlength{\arraycolsep}{0.6pt}\scalefont{.4}{\left[\begin{array}{rrrr|rrrr|rrrr|rrrr}
- & + & + & +&&&&&&&&&&&&\\
&&&&- & + & + & +&&&&&&&&\\
&&&&&&&&- & + & + & +&&&&\\
&&&&&&&&&&&&- & + & + & +\\
+ & - & + & +&&&&&&&&&&&&\\
&&&&+ & - & + & +&&&&&&&&\\
&&&&&&&&+ & - & + & +&&&&\\
&&&&&&&&&&&&+ & - & + & +\\
+ & + & - & +&&&&&&&&&&&&\\
&&&&+ & + & - & +&&&&&&&&\\
&&&&&&&&+ & + & - & +&&&&\\
&&&&&&&&&&&&+ & + & - & +\\
+ & + & + & -&&&&&&&&&&&&\\
&&&&+ & + & + & -&&&&&&&&\\
&&&&&&&&+ & + & + & -&&&&\\
&&&&&&&&&&&&+ & + & + & -\\
\end{array}\right]},\quad
M^{3,:,:}&\hspace{-3mm}=\setlength{\arraycolsep}{0.6pt}\scalefont{.4}{\left[\begin{array}{rrrr|rrrr|rrrr|rrrr}
- & + & + & +&&&&&&&&&&&&\\
&&&&- & + & + & +&&&&&&&&\\
&&&&&&&&- & + & + & +&&&&\\
&&&&&&&&&&&&- & + & + & +\\
&&&&+ & - & + & +&&&&&&&&\\
&&&&&&&&+ & - & + & +&&&&\\
&&&&&&&&&&&&+ & - & + & +\\
+ & - & + & +&&&&&&&&&&&&\\
&&&&&&&&+ & + & - & +&&&&\\
&&&&&&&&&&&&+ & + & - & +\\
+ & + & - & +&&&&&&&&&&&&\\
&&&&+ & + & - & +&&&&&&&&\\
&&&&&&&&&&&&+ & + & + & -\\
+ & + & + & -&&&&&&&&&&&&\\
&&&&+ & + & + & -&&&&&&&&\\
&&&&&&&&+ & + & + & -&&&&\\
\end{array}\right]}.
\end{aligned}\]
Finding the inner products of each basis we find the three Hadamard matrices

\[\begin{aligned}
H_{1,2}=\setlength{\arraycolsep}{.8pt}\scalefont{.4}{\left[\begin{array}{rrrrrrrrrrrrrrrr}
+&-&-&-&-&+&+&+&-&+&+&+&-&+&+&+\\
-&+&+&+&+&-&-&-&-&+&+&+&-&+&+&+\\
-&+&+&+&-&+&+&+&+&-&-&-&-&+&+&+\\
-&+&+&+&-&+&+&+&-&+&+&+&+&-&-&-\\
-&+&-&-&+&-&+&+&+&-&+&+&+&-&+&+\\
+&-&+&+&-&+&-&-&+&-&+&+&+&-&+&+\\
+&-&+&+&+&-&+&+&-&+&-&-&+&-&+&+\\
+&-&+&+&+&-&+&+&+&-&+&+&-&+&-&-\\
-&-&+&-&+&+&-&+&+&+&-&+&+&+&-&+\\
+&+&-&+&-&-&+&-&+&+&-&+&+&+&-&+\\
+&+&-&+&+&+&-&+&-&-&+&-&+&+&-&+\\
+&+&-&+&+&+&-&+&+&+&-&+&-&-&+&-\\
-&-&-&+&+&+&+&-&+&+&+&-&+&+&+&-\\
+&+&+&-&-&-&-&+&+&+&+&-&+&+&+&-\\
+&+&+&-&+&+&+&-&-&-&-&+&+&+&+&-\\
+&+&+&-&+&+&+&-&+&+&+&-&-&-&-&+\\
\end{array}\right]},\quad
H_{1,3}&=\setlength{\arraycolsep}{.8pt}\scalefont{.4}{\left[\begin{array}{rrrrrrrrrrrrrrrr}
+&-&-&-&-&+&+&+&-&+&+&+&-&+&+&+\\
-&+&+&+&+&-&-&-&-&+&+&+&-&+&+&+\\
-&+&+&+&-&+&+&+&+&-&-&-&-&+&+&+\\
-&+&+&+&-&+&+&+&-&+&+&+&+&-&-&-\\
+&-&+&+&-&+&-&-&+&-&+&+&+&-&+&+\\
+&-&+&+&+&-&+&+&-&+&-&-&+&-&+&+\\
+&-&+&+&+&-&+&+&+&-&+&+&-&+&-&-\\
-&+&-&-&+&-&+&+&+&-&+&+&+&-&+&+\\
+&+&-&+&+&+&-&+&-&-&+&-&+&+&-&+\\
+&+&-&+&+&+&-&+&+&+&-&+&-&-&+&-\\
-&-&+&-&+&+&-&+&+&+&-&+&+&+&-&+\\
+&+&-&+&-&-&+&-&+&+&-&+&+&+&-&+\\
+&+&+&-&+&+&+&-&+&+&+&-&-&-&-&+\\
-&-&-&+&+&+&+&-&+&+&+&-&+&+&+&-\\
+&+&+&-&-&-&-&+&+&+&+&-&+&+&+&-\\
+&+&+&-&+&+&+&-&-&-&-&+&+&+&+&-\\
\end{array}\right]},\quad
H_{2,3}=&\hspace*{-3mm}\setlength{\arraycolsep}{.8pt}\scalefont{.4}{\left[\begin{array}{rrrrrrrrrrrrrrrr}
+&-&-&-&+&-&+&+&+&+&-&+&+&+&+&-\\
-&+&+&+&-&+&-&-&+&+&-&+&+&+&+&-\\
-&+&+&+&+&-&+&+&-&-&+&-&+&+&+&-\\
-&+&+&+&+&-&+&+&+&+&-&+&-&-&-&+\\
+&+&+&-&+&-&-&-&+&-&+&+&+&+&-&+\\
+&+&+&-&-&+&+&+&-&+&-&-&+&+&-&+\\
+&+&+&-&-&+&+&+&+&-&+&+&-&-&+&-\\
-&-&-&+&-&+&+&+&+&-&+&+&+&+&-&+\\
+&+&-&+&+&+&+&-&+&-&-&-&+&-&+&+\\
+&+&-&+&+&+&+&-&-&+&+&+&-&+&-&-\\
-&-&+&-&+&+&+&-&-&+&+&+&+&-&+&+\\
+&+&-&+&-&-&-&+&-&+&+&+&+&-&+&+\\
+&-&+&+&+&+&-&+&+&+&+&-&+&-&-&-\\
-&+&-&-&+&+&-&+&+&+&+&-&-&+&+&+\\
+&-&+&+&-&-&+&-&+&+&+&-&-&+&+&+\\
+&-&+&+&+&+&-&+&-&-&-&+&-&+&+&+\\
\end{array}\right]}.\\
\end{aligned}\]
giving us a rank $v$ idempotent
\[M = \frac{1}{12}\left[\begin{array}{ccc}
4I & H_{12}& H_{13}\\
H_{12}^T & 4I & H_{23}\\
H_{13}^T& H_{23}^T & 4I
\end{array}\right].\]
Taking the positive entries of the off diagonal blocks of this matrix gives us the adjacency matrix of a $\mu$-heavy $\text{LSSD}(16,10,6;3)$.

\section*{Acknowledgments}
We are grateful to Padraig O Cathain and Bill Kantor for helpful conversations on these topics.

\section{Appendix 1: Feasible parameter sets}
\label{families}
The Handbook of Combinatorial Designs gives us a list of 21 distinct families of symmetric designs. We now examine each family to determine which parameter sets could be extended to LSSDs on three or more fibers. The two conditions we will employ are that $s=\sqrt{k-\lambda}$ and $\nu = \frac{k(k\pm s)}{v}$ are integers, though we will often reference our observations from \ref{gcd}. Our results show that Families 6, 7, 9, 12, 13, and 14 always permit integral parameters. Further, Families 15-19 give us integers in specific cases ($m=1$) but will not be feasible in general. It should be noted that this does not mean that we can find LSSDs in each of these families with $w>2$, instead this means that we cannot disprove the existence of such LSSDs using only our integrality conditions.
\subsection*{Family 1 (Point-hyperplane Designs)}
\[\begin{aligned}
v &= q^m+\dots+1 \qquad
k = q^{m-1}+\dots+1\qquad
\lambda = q^{m-2}+\dots+1\qquad n= q^{m-1}\qquad s=q^{\frac{m-1}{2}}\end{aligned}\]
Since $s$ is a power of $q$, we know that $\gcd(s,v) = 1$. Therefore we cannot form LSSDs from this family.
\subsection*{Family 2 (Hadamard matrix designs)}
\[\begin{aligned}
v &= 4n-1 \qquad k= 2n-1 \qquad \lambda = n-1\qquad s = \sqrt{n}\\
\end{aligned}\]
Since $s$ divides $v+1$, we know that $\gcd(s,v) = 1$ Therefore we cannot form LSSDs from this family.

\subsection*{Family 3 (Chowla)}
\[\begin{aligned}
v&=4t^2+1 \qquad k=t^2 \qquad \lambda = \frac{1}{4}(t^2-1)
\qquad s=\frac{1}{2}\sqrt{3t^2+1}
\end{aligned}\]
Chowla designs require that $v$ is prime, which is impossible for LSSDs.

\subsection*{Family 4 (Lehmer)}
\begin{enumerate}[(a)]
\item
  \[\begin{aligned}
  v&= 4t^2+9\qquad k=t^2+3\qquad \lambda = \frac{1}{3}(t^2+3)\qquad  n=\frac{3}{4}k
  \end{aligned}\]
\item
  \[\begin{aligned}
  v&=8t^2+1 = 64u^2+9 \qquad k=t^2\qquad \lambda=u^2\qquad  n=t^2-u^2\\
  \end{aligned}\]

\item
  \[\begin{aligned}
  v&=8t^2+49 = 64u^2+441\qquad k=t^2+6\qquad \lambda = u^2+7\qquad  n=t^2-u^2-1\\
  \end{aligned}\]
  All three of the Lehmer designs require $v$ to be prime, which is not possible for LSSDs.
\end{enumerate}

\subsection*{Family 5 (Whiteman)}
\[\begin{aligned}
v&=pq \qquad k=\frac{1}{4}(pq-1) \qquad \lambda = \frac{1}{16}(pq-5)\qquad s=\frac{1}{4}(3p+1)
\end{aligned}\]
where $q = 3p+2$. Since $\gcd(s,v) >1$ we must have $s = p$ or $s=q$. However $s=q$ implies that $p$ is negative while $s=p$ implies that $p=1$ and $q=5$. As this case gives the parameters $(5,1,0)$, only the degenerate case is possible. Therefore Whiteman parameters will never give us LSSDs.

\subsection*{Family 6 (Menon)}
\[\begin{aligned}
v &= 4t^2 \qquad k= 2t^2-t \qquad \lambda = t^2-t\\
n&= t^2 \qquad s = t\\
\nu &= \frac{(2t^2-t)(2t^2-t\pm t)}{4t^2}=\frac{1}{2}(2t-1)\left(t-\frac{1\mp1}{2}\right)
\end{aligned}\]
Since $2t-1$ will always be odd, we must have that $\left(t-\frac{1\mp1}{2}\right)$ is even. This means that for odd $t$, we must choose the $+$ so that we have $\nu = (2t-1)\frac{t-1}{2}$. If instead $t$ is even then we must choose the $-$ so that $\nu = (2t-1)\frac{t}{2}$. This means that Menon parameters are always feasible, though we must choose our parameters to be $\nu$-heavy or $\mu$-heavy based on the parity of $t$.

\subsection*{Family 7 (Wallis; McFarland)}
\[\begin{aligned}
v &= q^{m+1}(q^m+\dots+q+2) \qquad k= q^m(q^m+\dots+q+1) \qquad \lambda = q^m(q^{m-1}+\dots+q+1) \qquad s = q^m\\
\nu &= \frac{q^m(q^m+\dots+q+1)(q^m(q^m+\dots+q+1)\pm q^m)}{q^{m+1}(q^m+\dots+q+2)}\\
\end{aligned}\]
If we assume our parameters are $\nu$-heavy, then we have:
\[\begin{aligned}
\nu&=\frac{q^m(q^m+\dots+q+1)(q^m(q^m+\dots+q+2))}{q^{m+1}(q^m+\dots+q+2)}=q^{m-1}(q^m+\dots+q+1)
\end{aligned}\]

\subsection*{Family 8 (Wilson; Shrikhande and Singhi)}
\[\begin{aligned}
v&= m^3+m+1 \qquad k=m^2+1\qquad \lambda = m\qquad n=m^2-m+1\\
\end{aligned}\]
Note that $v = mk+1$. Therefore $\gcd(k,v) = 1$ and we cannot make LSSDs from these designs.

\subsection*{Family 9 (Spence)}
\[\begin{aligned}
v&= 3^m\left(\frac{3^m-1}{2}\right)\qquad k=3^{m-1}\left(\frac{3^m+1}{2}\right)\qquad \lambda = 3^{m-1}\left(\frac{3^{m-1}+1}{2}\right)\qquad 
s=3^{m-1}\\
\nu&=\frac{\frac{1}{2}3^{m-1}(3^m+1)(\frac{1}{2}3^{m-1}(3^m+1)\pm3^{m-1})}{ \frac{1}{2}3^m(3^m-1)}\end{aligned}\]
If we take $\mu$ heavy parameters, then
\[\begin{aligned}
\nu&=\frac{(3^m+1)\left(\frac{1}{2}3^{m-1}(3^m-1)\right)}{3(3^m-1)}=3^{m-2}\left(\frac{3^m+1}{2}\right)
\end{aligned}\]
\subsection*{Family 10 (Rajkundlia and Mitchell; Ionin)}
\[\begin{aligned}
v&=1+qr\left(\frac{r^m-1}{r-1}\right)\qquad k=r^m\qquad \lambda = r^{m-1}\left(\frac{r-1}{q}\right)\qquad r=\frac{q^d-1}{q-1}
\end{aligned}\]
Since $r$ divides $v-1$ and $k$ is a power of $r$, we know that $\gcd(v,k) = 1$. Therefore we cannot make LSSDs from these designs.
\subsection*{Family 11 (Wilson; Brouwer)}
\[\begin{aligned}
v&= 2(q^m+\dots + q)+1\qquad k=q^m\qquad \lambda=\frac{1}{2}q^{m-1}(q-1)\qquad n=\frac{1}{2}q^{m-1}(q+1)\\
\end{aligned}\]
Since $q$ divides $v-1$ and $k$ is a power of $q$, we must have that $\gcd(k,v) = 1$. Therefore we cannot make LSSDs from these designs.
\subsection*{Family 12 (Spence, Jungnickel and Pott, Ionin)}
\[\begin{aligned}
v&=q^{d+1}\left(\frac{r^{2m}-1}{r-1}\right)\qquad k=r^{2m-1}q^d\qquad\lambda = (r-1)r^{2m-2}q^{d-1}\qquad s=r^{m-1}q^d\qquad r=\frac{q^{d+1}-1}{q-1}\\
\nu&=\frac{r^{2m-1}q^d(r^{2m-1}q^d\pm r^{m-1}q^d)}{q^{d+1}\left(\frac{r^{2m}-1}{r-1}\right)} = \frac{q^{d-1}r^{3m-2}(r^{m}\pm 1)}{r^{2m-1}+\dots+1}
\end{aligned}\]
However $r^{3m-2}$ is relatively prime with the denominator, so we must have $(r^{2m-1}+\dots+1)\vert q^{d-1}\left(r^m\pm 1\right)$. However since $r = \frac{q^{d+1}-1}{q-1} = q^{d}+\dots+1$, we have that $q^{d-1}<r$. Therefore $q^{d-1}\left(r^m\pm 1\right)<r^{m+1}\pm r<r^{2m-1}\dots+1$ in all cases except $m=1$. However, when $m=1$,
\[\begin{aligned}
v&=q^{d+1}\left(q^d+\dots+q+2\right)\qquad k=q^d(q^d+\dots+1)\qquad\lambda=q^{d}\left(q^{d-1}+\dots+q+1\right)\qquad s=q^d
\end{aligned}\]
Giving us the same parameters as McFarland parameters (Family 7). Therefore these parameters will work for $\nu$-heavy designs when $m=1$.

\subsection*{Family 13 (Davis and Jedwab)}
\[\begin{aligned}
v&=\frac{1}{3}2^{2d+4}\left(2^{2d+2}-1\right)\qquad k=\frac{1}{3}2^{2d+1}\left(2^{2d+3}+1\right)\qquad \lambda = \frac{1}{3}2^{2d+1}\left(2^{2d+1}+1\right)\qquad s=2^{2d+1}\\
\nu&=\frac{\frac{1}{3}2^{2d+1}\left(2^{2d+3}+1\right)\left(\frac{1}{3}2^{2d+1}\left(2^{2d+3}+1\right)\pm 2^{2d+1}\right)}{\frac{1}{3}2^{2d+4}\left(2^{2d+2}-1\right)}\\
&=\frac{\left(2^{2d+3}+1\right)\left(\left(2^{2d+3}+1\right)\pm 3\right)2^{2d-2}}{3\left(2^{2d+2}-1\right)}\\
\end{aligned}\]
If we take $\mu$-heavy parameters, then we get:
\[\begin{aligned}
\nu&=\frac{\left(2^{2d+3}+1\right)\left(2^{2d+3}-2\right)2^{2d-2}}{3\left(2^{2d+2}-1\right)}=\frac{\left(2^{2d+3}+1\right)2^{2d-1}}{3}.
\end{aligned}\]
As $2^n + 1$ is divisible by 3 anytime $n$ is odd, this will always be an integer.

\subsection*{Family 14 (Chen)}
\[\begin{aligned}
v&=4q^{2d}\left(\frac{q^{2d}-1}{q^2-1}\right)\qquad k=q^{2d-1}\left(1+2\left(\frac{q^{2d}-1}{q-1}\right)\right)\qquad\lambda = q^{2d-1}(q-1)\left(\frac{q^{2d-1}+1}{q+1}\right)\qquad s=q^{2d-1}\\
\nu&=\frac{q^{2d-1}\left(1+2\left(\frac{q^{2d}-1}{q-1}\right)\right)\left(q^{2d-1}\left(1+2\left(\frac{q^{2d}-1}{q-1}\right)\right)\pm q^{2d-1}\right)}{4q^{2d}\left(\frac{q^{2d}-1}{q^2-1}\right)}\\
\end{aligned}\]
If we use $\mu$-heavy parameters, this gives us:
\[\begin{aligned}
\nu&=\frac{\left(1+2\left(\frac{q^{2d}-1}{q-1}\right)\right)\left(2q^{2d-1}\left(\frac{q^{2d}-1}{q-1}\right)\right)}{4q\left(\frac{q^{2d}-1}{q^2-1}\right)}=\frac{q^{2d-2}(q+1)\left(1+2\left(\frac{q^{2d}-1}{q-1}\right)\right)}{2}\\
\end{aligned}\]
Since $2$ will always divide either $q^{2d-2}$ or $q+1$, we have that $\nu$ is integral under $\mu$-heavy parameters

\subsection*{Family 15 (Ionin)}
\[\begin{aligned}
v&=q^d\left(\frac{r^{2m}-1}{(q-1)(q^d+1)}\right)\qquad k= q^dr^{2m-1}\qquad \lambda = q^d(q^d+1)(q-1)r^{2m-2}\qquad s=q^dr^{m-1}\\
r&=q^{d+1}+q-1\\
\nu&=\frac{q^dr^{2m-1}\left(q^dr^{2m-1}\pm q^dr^{m-1}\right)}{q^d\left(\frac{r^{2m}-1}{(q-1)(q^d+1)}\right)}=\frac{(q-1)(q^d+1)q^dr^{3m-2}\left(r^m\pm 1\right)}{(r^m+1)(r^m-1)}=\frac{(q-1)(q^d+1)q^dr^{3m-2}}{(r^m\mp1)}.\\
\end{aligned}\]
We now split these parameters into two cases, first when $m=1$ and then the remaining possibilities where $m\geq 2$.\par
First assume that $m=1$. Then,
\[\begin{aligned}
\nu&=\frac{(q-1)(q^d+1)q^dr}{(r\mp1)}.
\end{aligned}\]
Under $\mu$-heavy conditions, this gives
\[\begin{aligned}
\nu&=\frac{(q-1)(q^d+1)q^dr}{(r+1)}=\frac{(q-1)(q^d+1)q^dr}{q(q^d+1)}=(q-1)q^{d-1}r
\end{aligned}\]
Therefore these parameters are feasible using $\mu$-heavy parameters when $m=1$. \par
Now consider when $m\geq 2$. In this case, note that $r^{3m-2}$ is relatively prime to $r^m\mp 1$. Therefore if $\nu$ is integral, then $r^m\mp 1$ must divide $q^d(q-1)(q^d+1)$. However, since $q\geq 2$ we know that $r = q(q^d+1)-1>q^d+1$ and $r = q^{d+1}+q-1>q^{d+1}-q^d$. Therefore $r^m\geq r^2 > q^d(q-1)(q^d+1)$ meaning that it is not possible for $r^m$ to divide the latter. Therefore $\nu$ will never be integral when $m>1$.

\subsection*{Family 16 (Ionin)}
\[\begin{aligned}
v&=2\cdot3^d\left(\frac{q^{2m}-1}{3^d+1}\right)\qquad k= 3^dq^{2m-1}\qquad\lambda = \frac{1}{2}3^d(3^d+1)q^{2m-2}\qquad s = 3^dq^{m-1}\qquad q=\frac{1}{2}(3^{d+1}+1)\\
\nu&=\frac{3^dq^{2m-1}(3^dq^{2m-1}\pm3^dq^{m-1})}{2\cdot3^d\left(\frac{q^{2m}-1}{3^d+1}\right)}=\frac{3^d(3^d+1)q^{3m-2}(q^{m}\pm 1)}{2(q^m+1)(q^m-1)}=\frac{3^d(3^d+1)q^{3m-2}}{2(q^m\mp1)}\\
\end{aligned}\]
We again must consider the case when $m=1$ seperately. If $m=1$, then
\[\begin{aligned}\nu&=\frac{3^d(3^d+1)q}{2(q\mp1)}\\
\end{aligned}\]
If we take $\mu$ heavy parameters, we have
\[\begin{aligned}
\nu&=\frac{3^d(3^d+1)q}{3^{d+1}+3}=3^{d-1}q\\
\end{aligned}\]
Therefore when $m=1$, these parameters are feasible with $\mu$-heavy parameters. Using the same arguments as before, we can quickly find that $\nu$ will not be an integer for $m>1$ noting that $q$ is relatively prime to $q^m\mp 1$ and $q^m-1> 3^d(3^d+1)$.

\subsection*{Family 17 (Ionin)}
\[\begin{aligned}
v&=3^d\left(\frac{q^{2m}-1}{2(3^d-1)}\right)\qquad k= 3^dq^{2m-1}\qquad\lambda = 23^d(3^d-1)q^{2m-2}\qquad s=3^dq^{m-1}\qquad q=3^{d+1}-2\\
\nu&=\frac{3^dq^{2m-1}\left(3^dq^{2m-1}\pm 3^dq^{m-1}\right)}{3^d\left(\frac{q^{2m}-1}{2(3^d-1)}\right)}=\frac{3^dq^{3m-2}\left(q^m\pm 1\right)\left(2(3^d-1)\right)}{\left(q^{2m}-1\right)}=\frac{2q^{3m-2}3^d(3^d-1)}{\left(q^{m}\mp1\right)}\\
\end{aligned}\]
As before, consider the case when $m=1$ under $\nu$-heavy parameters,
\[\begin{aligned}
\nu&=\frac{2q^{3m-2}3^d(3^d-1)}{\left(q-1\right)}=\frac{2q^{3m-2}3^d(3^d-1)}{\left(3^{d+1}-3\right)}=2q^{3m-2}3^{d-1}.\\
\end{aligned}\]
Therefore these parameters are feasible when $m=1$. We again find that $m\geq 2$ will not permit $\nu$ to be an integer as $q^{3m-2}$ is relatively prime to $q^m\pm 1$ and $q^m-1>2\cdot3^d(3^d-1)$.

\subsection*{Family 18 (Ionin)}
\[\begin{aligned}
v&=2^{2d+3}\left(\frac{q^{2m}-1}{q+1}\right)\qquad k= 2^{2d+1}q^{2m-1}\qquad \lambda = 2^{2d-1}(q+1)q^{2m-2}\qquad s=2^{2d+1}q^{m-1}\qquad q=\frac{1}{3}\left(2^{2d+3}+1\right)\\
\nu&=\frac{2^{2d+1}q^{2m-1}\left(2^{2d+1}q^{2m-1}\pm 2^{2d+1}q^{m-1}\right)}{2^{2d+3}\left(\frac{q^{2m}-1}{q+1}\right)}=\frac{(q+1)2^{2d-1}q^{3m-2}}{\left(q^{m}\mp1\right)}\\
\end{aligned}\]
If $m=1$ and we take $\mu$-heavy parameters, then $\nu = 2^{2d-1}q^{3m-2}$. As before, $\nu$ is non integral when $m>1$ noting that $q^{3m-2}$ is relatively prime to $q^m\mp 1$ and $q^m-1>(q+1)2^{2d-1}$.

\subsection*{Family 19 (Ionin)}
\[\begin{aligned}
v&=2^{2d+3}\left(\frac{q^{2m}-1}{3q-3}\right)\qquad k = 2^{2d+1}q^{2m-1} \qquad \lambda = 3*2^{2d-1}(q-1)q^{2m-2}\qquad s = 2^{2d+1}q^{m-1}\qquad q=2^{2d+3}-3\\
\nu&=\frac{2^{2d+1}q^{2m-1}\left(2^{2d+1}q^{2m-1}\pm 2^{2d+1}q^{m-1}\right)}{2^{2d+3}\left(\frac{q^{2m}-1}{3q-3}\right)}=\frac{2^{2d-1}q^{3m-2}3(q-1)}{\left(q^m\mp1\right)}\\
\end{aligned}\]
If $m=1$ and we take $\nu$-heavy parameters then $\nu = 2^{2d-1}3q$. As before, $\nu$ is non integral when $m>1$ noting that $q^{3m-2}$ is relatively prime to $q^m\mp 1$ and $q^m\mp1>3\cdot2^{2d-1}(q-1)$.

\subsection*{Family 20 (Ionin)}
For this family we use the only known realization where $p=2$ and $q=2^d-1$ is a Mersenne prime.
\[\begin{aligned}
v&=1+2^{d+1}\frac{2^{dm}-1}{2^d+1}\qquad k=2^{2dm} \qquad \lambda = 2^{2dm-d-1}(2^{d}+1)\qquad n=2^{2dm-d-1}(2^d-1)\\
\end{aligned}\]
Our first restriction tells us that $n$ must be a square. However since $2$ does not divide $2^d+1$, we know that $2^d-1$ must be a square in order for $n$ to be a square giving us a contradiction.
\subsection*{Family 21 (Kharaghani and Ionin)}
\[\begin{aligned}
v&=4t^2\left(\frac{q^{m+1}-1}{q-1}\right)\qquad k=(2t^2-t)q^m\qquad \lambda = (t^2-t)q^m\qquad s=tq^{\frac{m}{2}}\qquad q=(2t-1)^2\\
\nu &=\frac{(2t^2-t)q^m\left((2t^2-t)q^m\pm tq^{\frac{m}{2}}\right)}{4t^2\left(\frac{q^{m+1}-1}{q-1}\right)}\\
&=\frac{(2t-1)^{3m+1}\left((2t-1)^{m+1}\pm 1\right)(q-1)}{4\left((2t-1)^{2m+2}-1\right)}\\
&=\frac{(2t-1)^{3m+1}(q-1)}{4\left((2t-1)^{m+1}\mp1\right)}\\
\end{aligned}\]
First, since $(2t-1)$ is odd, we have that $(2t-1)^{3m+1}$ is relatively prime to $4((2t-1)^{m+1}\mp 1)$. However since $m\geq1$, $4((2t-1)^{m+1}\mp1) \geq 4(q-1)$ and thus $\nu$ is not integral. 
\subsection*{Summary}
We have shown here that only Families 6, 7, 9, 13, and 14 will always satisfy our integrality conditions. Further, Families 12, 15, 16, 17, 18, and 19 satisfy our integrality conditions whenever $m=1$. Finally all remaining families will not allow for any LSSDs with $w>2$. 

\end{document}